\documentclass[a4paper,18pt]{article}
\usepackage{amsmath}
\usepackage{indentfirst}
\usepackage[english]{babel}
\usepackage{graphics}
\usepackage{amssymb}
\usepackage{dsfont}
\usepackage{amsfonts}
\usepackage{amsthm}
\usepackage[dvips]{graphicx}
\usepackage[latin1]{inputenc}
\usepackage{enumerate}
\usepackage{amsxtra}
\usepackage{amstext}
\usepackage{amssymb}
\usepackage{latexsym}
\usepackage[pdftex,colorlinks]{hyperref}

\newcommand{\N}{\mathbb N}

\newcommand{\R}{\mathbb R}
\newcommand{\E}{\mathbb E}

\newcommand{\p}{\mathbb P}
\newtheorem{prop}{Proposition}[section]

\newtheorem{defi}[prop]{Definition}
\newtheorem{dfprop}[prop]{Definition/Proposition}
\newtheorem{lemme}[prop]{Lemma}
\newtheorem{remarque}[prop]{Remark}
\newtheorem{question}[prop]{Question}

\newtheorem{theo}[prop]{Theorem}

\newcommand{\Z}{\mathbb{Z}}

\title{Transience of algebraic varieties in linear groups - applications to  generic Zariski density}
\author{Richard Aoun \footnote{Laboratoire de Math\'ematiques, B\^atiment 425, Universit\'e Paris Sud 11, 91405 Orsay-
FRANCE, \newline E-mail: richard.aoun@math.u-psud.fr}}
\date{}
\begin{document}

\maketitle

\begin{abstract}
We study the transience of algebraic varieties in linear groups. In particular, we show that a ``non elementary''
 random walk in $SL_2(\R)$ escapes exponentially fast from every
proper algebraic  subvariety. We also treat the case where the random walk is on the real points of a semi-simple split algebraic group and show such a result   for a wide family of random walks.\\
As an application, we prove that generic subgroups (in some sense) of linear groups are Zariski dense.
\end{abstract}
\tableofcontents

\section{Introduction}

One of the essential results in probability theory on groups is   Kesten's theorem \cite{Kesten}:  the probability of return to identity of a random walk  on a group $\Gamma$ decreases exponentially fast if and only if $\Gamma$ is non amenable.
A natural question is to extend this to other subsets: for which subsets does the random walk
escape with exponential rate? Many authors has studied the case where the subset is a subgroup of $\Gamma$: see for example \cite{em1}, \cite{bekk} and in particular \cite[Theorem 51]{lin} where it is shown that the probability that a random walk on  $\Gamma$ returns to a subgroup $H$ decreases exponentially fast to zero if and only if the Scheirer graph of $\Gamma /H$ is non amenable. \\
In this note we look at random walks on Zariski dense subgroups of algebraic groups (such as $SL_2(\R)$)  and we look at the escape from  proper algebraic subvarieties. Such questions have an interest in their own right since they allow us to study the delicate behavior
of the random walk  but they have also been recently involved in other domains such as the theory of expander graphs.
 We are referring here among others to the works of Bourgain and Gamburd
 \cite{bg1},\cite{bg}, Breuillard and Gamburd
\cite{gambbreui}  and Varju \cite{Varju}. In \cite{gambbreui} for instance
it is shown that there is an infinite set of primes $p$
  of density one, such that
the family of all Cayley graphs of $ SL_2(\Z/p\Z)$ is a family of
 expanders.
A crucial part of the proof is to take a random
walk on $ SL_2(\Z/p\Z)$ and to show that the probability of remaining in a subgroup
decreases exponentially fast to zero and uniformly.  In \cite[Corollary 1.1.]{bg} the following statement was established:
consider the group $SL_d(\Z)$ ($d\geq 2$), the
uniform probability measure on a finite symmetric generating set and $(S_n)_{n\in \N}$ the associated random walk, then for every proper algebraic  variety $\mathcal{V}$ of $SL_d(\mathbb{C})$,  $\p (S_n \in \mathcal{V}) $
decreases exponentially fast to zero. \\
Kowalski \cite{Kowalski} and Rivin \cite{Rivin} were interested in similar questions: for example they were able to estimate the probability that
a random walk in $SL_d(\Z)$  lies in the set of matrices with reducible characteristic polynomial. The techniques used by Kowalski and Rivin are  arithmetic sieving ones.\\

In this article,  we develop  a more probabilistic approach allowing us to  deal with random walks on arbitrary Zariski dense subgroups of  semi-simple algebraic groups. In the particular case of $SL_2(\R)$, we obtain (see Theorem \ref{ch2tr}) that a random walk whose measure generates a non-elementary subgroup escapes with probability tending to one exponentially fast  from every algebraic variety.  Our method relies on the theory of random matrix products developed in the
60's by Kesten and Furstenberg and in the 70's-80's by the French school: in particular Bougerol, Guivarc'h, Le Page and Raugi.\\

We also apply our techniques to generic Zariski density. Let $\Gamma_1$ and $\Gamma_2$ be two Zariski dense subgroups of $SL_d(\R)$ ($d\geq 2$). We prove in Theorem \ref{ch2generic1} that one can exhibit a probability measure on each of the subgroups such that two independent random walks will eventually generate  a Zariski dense subgroup. We have proved in  \cite{aoun} that the latter subgroup is also free. This gives consequently a ``probabilistic'' version of the Tits alternative \cite{tits}.\\


 All the random variables will be defined on a probability space $(\Omega, \mathcal{F},\p)$, the symbol $\E$
 will refer to the expectation with respect to $\p$ and ``a.s.'' to almost surely. If $\Gamma$ is a topological group, $\mu$ a probability measure on $\Gamma$, we define a sequence of independent random variables $\{X_n; n\geq 0\}$ with the same law $\mu$. We denote for every $n\in \N^*$ by $S_n=X_n \cdots X_1$ the $n^{th}$ step of the random walk. \\
First let us present the result we obtain for $SL_2(\R)$. We will say that  a probability measure $\mu$ on $SL_2(\R)$ is non elementary if the group  generated by its support
 is non elementary, i.e. Zariski dense in $SL_2(\R)$ or equivalently not virtually solvable.
\begin{theo} Let $\mu$ be a non elementary probability measure on $SL_2(\R)$ having an exponential moment (see Section \ref{ch2subprelprob} for a definition of this notion).
 Then for every proper algebraic subvariety $\mathcal{V}$ of $SL_2(\R)$, $$\displaystyle \limsup_{n\rightarrow \infty} \big[\p (S_n \in \mathcal{V})\big]^{\frac{1}{n}} < 1$$
 In particular, every proper  algebraic  subvariety is transient, that is a.s. $S_n$ leaves $\mathcal{V}$ after some time. \\
More precisely, if $P$ is a non constant polynomial equation in the entries of the $2\times 2$ matrices of $SL_2(\R)$, then there exists $\lambda>0$ such that:
$$\frac{1}{n} \log |P(S_n)| \underset{n\rightarrow \infty}{\overset{\textrm {a.s.}}{\longrightarrow}} \lambda$$
A large deviation inequality  holds as well:
 for every $\epsilon>0$:
\begin{equation}\displaystyle \limsup_{n\rightarrow \infty} \Big[\p \left( \big|\frac{1}{n} \log |P(S_n)| - \lambda\big| > \epsilon \right) \Big]^{\frac{1}{n}}< 1\label{ch2lgdevia}\end{equation}
\label{ch2tr}\end{theo}


Theorem \ref{ch2tr} is in fact a particular case of a more general statement: Theorem \ref{ch2tr1} below.
If $G$ is the group of real points of an algebraic group $\mathbf{G}$, $m$ a Cartan projection (see Section \ref{ch2subprel}), $\mu$ a probability measure on $G$, then the Kingman subadditive ergodic theorem allows us to define a vector $Liap(\mu)$ (see Definition / Proposition \ref{ch2liapuvector}) in the Weyl chamber of $G$ which is the almost sure limit of $\frac{1}{n}m(S_n)$
.
\begin{theo}\label{ch2tr1} Let $\mathbf{G}$ be a  semi-simple algebraic group defined and split over $\R$\footnote{For example, $\mathbf{G}=\mathbf{SL}_d$, $d\geq 2$.}, $G=\mathbf{G}(\R)$ its group of  real points, $\Gamma$ a Zariski dense subgroup
of $G$, $\mathcal{V}$ a proper algebraic subvariety of $\mathbf{G}$ defined over $\R$, $\mu$ a probability on $G$ with  an exponential moment (see Section \ref{ch2subprelprob}) such that its support generates  $\Gamma$. Then, there exists a finite union of hyperplanes $H_1,\cdots, H_r$ in the Weyl chamber (see Section \ref{ch2subcartan}) depending only on $\mathcal{V}$ such that if $Liap(\mu)\not \in H_1 \cup \cdots \cup H_r$ then, \begin{equation}\displaystyle \limsup_{n\rightarrow \infty} \big[\p (S_n \in \mathcal{V})\big]^{\frac{1}{n}} < 1\label{ch2olli1}\end{equation}
Probability measures, whose support generates  $\Gamma$,  satisfying the condition $Liap(\mu) \not\in H_1 \cup \cdots H_r$ exist (See Lemma \ref{ch2lyapunovcone}). A large deviation inequality similar to (\ref{ch2lgdevia})
holds as well.

\end{theo}

 Theorem \ref{ch2tr1} clearly implies Theorem \ref{ch2tr}: indeed, everything we want to show is that the Lyapunov exponent
associated to $\mu$ (see Definition \ref{ch2liapou}) is non zero (positive). This is ensured by Furstenberg's  theorem \cite{Furst}. \\

\begin{remarque} The number $\lambda$ that appears in Theorem \ref{ch2tr} or \ref{ch2tr1},
should be seen as a generalization of the classical Lyapunov exponent (see Definition \ref{ch2liapou}). In fact, it will be the Lyapunov exponent relative to the probability measure $\rho(\mu)$ where $\rho$ is some rational representation of $\mathbf{G}$.
\end{remarque}

\begin{remarque}
Our method doesn't allow us to estimate $\p(S_n\in \mathcal{V})$ when $Liap(\mu)$ belongs to the finite union of hyperplanes $H_i$ defined by the variety $\mathcal{V}$.  Example 2 of Section \ref{ch2example} illustrates this. \end{remarque}

Let us justify why we will look at the escape from algebraic subvarieties  and not from
$C^1$ submanifolds for instance.  Kac and Vinberg  proved in   \cite{vinberg} (see also \cite{benoist}) that there exist discrete Zariski dense subgroups of
$SL_3(\R)$ preserving a $C^1$ (but not algebraic) manifold on the projective plane (in fact, such manifolds are obtained as the
boundary of a divisible convex in $P^2(\R)$). Let $\Gamma$ be such a group, $\mathcal{C}$ such a manifold
and $\mathcal{V}=\{x\in \R^3\setminus\{0\}; [x] \in \mathcal{C}\}\cup\{0\}$ where $[x]$ denotes the projection of $x\neq 0$
on $P^2(\R)$. Note that $\mathcal{V}$ is differentiable outside $0$.
Then, for every $x\in \mathcal{V}$, every $n\in \N$, $\p (S_n x \in \mathcal{V}) = 1$. By way of contrast, we  show in the following statement that for proper algebraic  subvarieties the latter quantity decreases exponentially fast to zero.\\

 \begin{theo}\label{ch2tr2} Let $\Gamma$ be a Zariski dense subgroup of $SL_d(\R)$ ($d\geq 2$), $\mu$ a probability measure with an exponential moment whose support generates  $\Gamma$. Then for every proper algebraic subvariety $\mathcal{V}$ of $\R^d$, every non zero vector $x$ of $\R^d$ we have:

$$\displaystyle \limsup_{n\rightarrow \infty} \big[\p (S_n x \in \mathcal{V})\big]^{\frac{1}{n}} < 1$$\end{theo}

As discussed at the beginning of the introduction, it is interesting to study the transience of proper subgroups. It follows from Varju's paper  (see \cite[Propositions 8 and 9]{Varju}) that if $\mathbf{E}$ is a simple algebraic group defined over $\R$, $\mathbf{G}$ the direct product of $r$ copies of $\mathbf{E}$ (with $r\in \N^*$), $\Gamma$ a Zariski dense subgroup of $G=\mathbf{G}(\R)$, then there exists a symmetric probability measure $\mu$ on $\Gamma$  whose support generates $\Gamma$ such that the probability that the associated random walk  escapes from a proper algebraic  subgroup decreases exponentially fast to zero.




We will show that this in fact holds for all probability measures with an exponential moment whose support generates $\Gamma$ and for every semi-simple algebraic group $\mathbf{G}$, namely:
\begin{theo}\label{ch2tr3} Let $\mathbf{G}$ be a semi-simple algebraic group defined over $\R$,  $G$ its group of real points assumed without compact factors, $\Gamma$ a Zariski dense subgroup of $G$ and $\mu$  a probability measure  with an exponential
moment whose support generates $\Gamma$. Then for every proper algebraic subgroup $\mathbf{H}$ of $\mathbf{G}$,
$$\displaystyle \limsup_{n\rightarrow \infty} \big[\p (S_n \in H)\big]^{\frac{1}{n}} < 1$$
where $H$ is the group of real points of $\mathbf{H}$. \end{theo}

The bound obtained by Varju is uniform over the subgroups. Unfortunately our bound in Theorem \ref{ch2tr3} is not.\\

Our estimates will be applied to show that Zariski density in linear groups is generic in the following sense:

\begin{theo}\label{ch2generic11} Let $G$ be the group of  real points of a semi-simple algebraic group split over $\R$. Let $\Gamma_1,\Gamma_2$ be two Zariski dense subgroups of $G$.
 Then there exist probability measures $\mu_1$ and $\mu_2$ with an exponential moment whose support generate respectively  $\Gamma_1$ and $\Gamma_2$  such that for some $c\in ]0,1[$ and all large $n$,
 $$\p (\textrm{$\langle S_{1,n},S_{2,n}\rangle$ is  Zariski dense and free}) \geq 1-c^n$$ where  $\{S_{2,n}; n \geq 0\}$
and  $\{S_{2,n}, n \geq 0\}$ are two independent random walks on $\Gamma_1$ (resp. $\Gamma_2$) associated respectively to
$\mu_1$ and $\mu_2$ on $\Gamma_1$ (resp. $\Gamma_2$). This implies that almost surely, for $n$ big enough, the subgroup
 $\langle S_{1,n},S_{2,n}\rangle$ is Zariski dense and free.
\end{theo}
See Section \ref{ch2subzariski} for the comparison of these results with Rivin's  in \cite{genericrivin}.

\begin{remarque} The fact that $\{w\in \Omega; \langle M_n(w), {M_n}'(w) \rangle \textrm{\;is Zariski dense}\}$ is measurable will follow from Lemma \ref{ch2strongtits}. \end{remarque}
\subsection{Outline of the paper}
In order to prove Theorem \ref{ch2tr1} (or \ref{ch2tr2}, \ref{ch2tr3}),
one can clearly suppose that $\mathcal{V}$ is a proper hypersurface
(i.e. the common zeroes of one polynomial equation). We will do so in
all the paper.\\

In Section \ref{ch2example}, we provide two examples to explain the general idea of the proofs.\\

Section \ref{ch2sublinear} is purely algebraic. To every proper algebraic hypersurface $\mathcal{V}$ of $\mathbf{G}$ we associate a rational real representation $\rho$ of $\mathbf{G}$ such that $g\in \mathcal{V}$ is equivalent to: the matrix coefficients of
$\rho(g)$ satisfy a linear condition ``$(L)$''. Thus we have ``linearized'' our variety. This can
 be seen as a generalization of the well-known Chevalley theorem
(Theorem \ref{ch2chevalley}) concerning the particular case of subgroups.

In Section \ref{ch2subprel} we recall standard facts about semi-simple algebraic groups and their rational representations.\\

In Section \ref{ch2subproba} we give some additional results to the theory of random matrix products. They  will be used in Section \ref{ch2subproof1} in order to show that $\rho(S_n)$ may verify $(L)$ only with a probability decreasing exponentially in $n$.\\

We consider a random walk on a Zariski dense subgroup $\Gamma$ of the real points of a semi-simple algebraic group. First we define the Lyapunov vector, which is the normalized Cartan projection of the random walk. We recall in Theorem \ref{ch2guimo} that it belongs to the interior of the Weyl chamber. In lemma \ref{ch2lyapunovcone}, we show that for every finite union of hyperplanes in the Weyl chamber, one can always find a probability measure whose support generates  $\Gamma$ such that the Lyapunov vector does not belong to this union (this is the condition stated in Theorem \ref{ch2tr1}).

Next, we will be interested in the behavior of the components of the random walk in the Cartan decomposition. In Theorems \ref{ch2ratioA} and \ref{ch2conv}, we give  new and shorter proof  of the exponential convergence in the $KAK$ decomposition we obtained in our previous work \cite{aoun}. Unlike \cite{aoun} when we were working on an arbitrary local field, we will take advantage during the proofs of the fact that our matrices are real valued.

Theorem \ref{ch2ratioA} shows the exponential decay of the ratio between the first two $A$-components of the random walk in
the KAK decomposition. This is a version in expectation of the fact that the Lyapunov vector  belongs to the interior of the Weyl chamber. The proof  will follow easily from a large deviations theorem of Le Page in $GL_d(\R)$. We note that we proved a similar result in \cite{aoun} but with different techniques, the reason is that a large deviation result over an arbitrary local field is not present in the literature. \\

Theorem  \ref{ch2conv} establishes the exponential convergence of the $K$-parts.\\


In Section \ref{ch2subproof1}, we prove our mains results: Theorems \ref{ch2tr1},  \ref{ch2tr2} and \ref{ch2tr3}. The key is Theorem \ref{ch2theo1} which
computes the probability that a random walk on a linear algebraic group verifies a linear condition on the matrix coefficients. No irreducibility assumptions are made, a genericity condition on the geometry of the Lyapunov vector is however needed.  \\

Finally in Section \ref{ch2subzariski}, we apply Theorem \ref{ch2theo1} to prove Theorem \ref{ch2generic11}. We compare our results with Rivin's  in \cite{genericrivin}.\\


\paragraph{Acknowledgments} I sincerely thank  Emmanuel Breuillard and Yves Guivarc'h for fruitful discussions, remarks and advices. I thank also Igor Rivin for his interest and his comments.

\section{Examples}\label{ch2example}
In this section, we give examples to illustrate the ideas and methods we will use in the next section to prove our main results.
\subsection{Example 1}
This example illustrates Theorem \ref{ch2tr2}.\\
Let $\Gamma$ be Zariski dense subgroup of $SL_3(\R)$ ($SL_3(\Z)$ for example).
 Consider a probability measure $\mu$ on $SL_3(\R)$ with an exponential moment (see Section \ref{ch2subprelprob}) whose support generates $\Gamma$. For example, if $\Gamma$ is finitely generated,
 choose a probability measure whose support is a finite symmetric generating set.  Let $S_n=X_n\cdots X_1$ be the associated  random walk. We write $S_n$ in the canonical basis of $M_{3,3}(\R)$:
$$S_n=\left(
       \begin{array}{ccc}
         a_n & b_n & c_n \\
         d_n & e_n & f_n \\
         g_n & h_n & i_n \\
       \end{array}
     \right)$$
     We propose to see if the following probability decreases exponentially fast to zero:  $$p_n=\p (a_n^2 - a_ne_n +2a_nd_n - a_nb_n -b_nd_n=0)$$

In other words if $\mathcal{V}$ is the
 proper algebraic hypersurface of $SL_3(\R)$ defined by $\mathcal{V}=\{\left( \begin{array}{ccc}
                        a & b & c \\
                        d & e & f \\
                        g & h & i
                      \end{array}\right)\in \Gamma; a^2-ae+2ad-ab-bd=0\}$, then we are interested in estimating  $\p(S_n\in \mathcal{V})$.

\paragraph{Step 1: Linearization of the algebraic hypersurface $\mathcal{V}$.\\}
Let $E$ be the vector space of homogenous polynomials on three variables $X,Y,Z$ of degree $2$. The group $SL_3(\R)$ acts on $E$ by the formula: $g \cdot P(X,Y,Z)=P\left( g^t (X,Y,Z) \right)$ where $g^t$ is the transposed matrix of $g$ when $g$ is expressed in the canonical basis. Let us write down this representation. We will consider the basis $\{X^2, Y^2,Z^2, XY, XZ, XY\}$ of $E$.

\begin{eqnarray}
SL_3(\R) & \overset{\rho}{\longrightarrow} & GL(E) \simeq GL_6(\R)\nonumber\\
\left(
  \begin{array}{ccc}
    a & b & c \\
    d & e & f \\
    g & h & i \\
  \end{array}
\right) & \mapsto & \left(
                      \begin{array}{cccccc}
                        a^2 & b^2 & c^2 & ab & ac & bc \\
                        d^2 & e^2 & f^2 & de & df & ef \\
                        g^2 & h^2 & i^2 & gh & gi & hi \\
                        2ad & 2be & 2cf & ae+bd & af+cd & bf+ec \\
                        2ag & 2bh & 2ci & ah+gb & ai+cg & bi+ch \\
                        2dg & 2eh & 2fi & dh+eg & di+gf & ei+hf \\
                      \end{array}
                    \right)\nonumber \end{eqnarray}

In what follows we identify $E$ with $\R^6$ by sending $\{X^2,Y^2,XY,XZ,YZ\}$ to the canonical basis $\{e_i; i=1, \cdots, 6\}$.
Then it is clear that $$\mathcal{V}=\{g\in SL_3(\R); \rho(g) (e_1-e_4)\in H\}$$ where $H$ is the hyperplane in $E$ defined by
$H=\{x=(x_i)_{i=1}^6\in \R^6; x_1+x_4=0\}$.\\
We say that we have linearized the hypersurface $\mathcal{V}$.  This method generalizes easily and yields Lemma \ref{ch2lemma} which holds for arbitrary hypersurfaces.\\
Note that, for $x=e_1-e_4$, $$p_n=\p \left( \rho(S_n) x \in H \right)$$

\paragraph{Random matrix products  in $GL_6(\R)$\\}
We have now a probability measure $\rho(\mu)$, image of $\mu$ under $\rho$, on $GL_6(\R)$ with an exponential moment.
The smallest closed group $G_{\rho(\mu)}$ containing the support of $\rho(\mu)$ is a Zariski dense subgroup of $\rho(SL_3(\R))$. One can verify that $\rho$ is in fact $SL_3(\R)$-irreducible. Since $SL_3(\R)$ is Zariski connected, we deduce that
$G_{\rho(\mu)}$ is a strongly irreducible (Definition \ref{defdef}) subgroup of $GL_6(\R)$.
 Moreover, the group $\rho\left(SL_3(\R)\right)$ contains clearly a proximal element,
 then by Goldsheild-Margulis theorem \cite{Margulis} (see Theorem \ref{ch2margulis} for the statement), the same applies for $G_{\rho(\mu)}$.\\
Thus, we can use the theory of random matrix products which gives (see Lemma \ref{ch2hyperplane})
what we wanted to prove, i.e.:
$$\limsup_{n\rightarrow +\infty} \frac{1}{n} \log {\p \left(\rho(S_n)x \in H \right)}< 0$$

A word about the proof: if $[x]$ denote the projection of $x\in \R^6 \setminus\{0\}$ in the projective space
$P(\R^6)$, then $\rho(S_n)[x]$ converges in law towards a random variable $Z$
with law the unique $\mu$-invariant probability measure $\nu$ on the projective space $P(\R^6)$. It can be shown that the speed of convergence is exponential in a certain sense. Moreover, almost surely,
$Z$ cannot belong to the hyperplane $H$ because $\nu$ is proper. More precisely, we can
control the probability that the  distance between $Z$ and a fixed hyperplane $H$ be small.

 \begin{remarque}
This method  does not give an estimate of the growth of $Q(S_n)$ where $Q$ is the polynomial that defines $\mathcal{V}$.
 We will see in the next section  (Theorem \ref{ch2theo1}) how  such quantities can be estimated. \end{remarque}

 \subsection{Example 2}
 This example illustrates  situations in which we are unable to obtain the exponential decrease of the probability
  of lying in a subvariety for all probability measures (see the statement of Theorem \ref{ch2tr1}).\\
 As in Example 1, consider a probability measure on $SL_3(\R)$ with an exponential moment whose support generates a Zariski dense subgroup of $SL_3(\R)$. Say that we would like to estimate the following probability:

 $$q_n=\p (a_ne_n-b_nd_n+2e_n=0)$$
 Let $\mathcal{S}$ be the following hypersurface of $SL_3(\R)$: $\mathcal{S}=\{ae-bd+2e=0\}$ so that $q_n=\p(S_n\in \mathcal{S})$.
 Consider the natural action of $SL_3(\R)$ on $F=\bigwedge^2 \R^3 \oplus \R^3$. Denote by $\eta$ this representation and  write $\eta=\eta_1\oplus \eta_2$.  We fix the basis $(e_1\wedge e_2, e_1 \wedge e_3, e_2\wedge e_3, e_1,e_2,e_3)$ of $F$.  Formally, we have:
 \begin{eqnarray}
SL_3(\R) & \overset{\eta}{\longrightarrow} & GL(F) \simeq GL_6(\R)\nonumber\\
\left(
  \begin{array}{ccc}
    a & b & c \\
    d & e & f \\
    g & h & i \\
  \end{array}
\right) & \mapsto & \left(
                      \begin{array}{cccccc}
                        ae-bd & af-cd & bf-ec & 0 & 0 & 0 \\
                        ah-gb & ai-gc & bi-hc & 0 & 0 & 0 \\
                        dh-eg & di-gf & ei-hf & 0 & 0 & 0 \\
                        0 & 0 & 0 & a & b & c \\
                        0 & 0 & 0 & d & e & f \\
                        0 & 0 & 0 & g & h & i \\
                      \end{array}
                    \right)\nonumber \end{eqnarray}
Thus $$\mathcal{S}=\{g\in SL_3(\R); \eta(g)x\in H\}$$
where $x=e_1\wedge e_2+e_2$ and $H=\{x\in \R^6; x_1+2x_5=0\}$.
Hence, we have linearized our variety $\mathcal{S}$ as in Example 1. The difference between these two examples  is that the representation $\eta$ is no longer irreducible ($\eta_1$ and $\eta_2$ are its irreducible sub-representations). Hence we cannot use the same argument as in Example 1.\\
However, we will see in the proof of Theorem \ref{ch2theo1} that we are able to solve the problem if the top Lyapunov exponents of $\eta_1(\mu)$ and $\eta_2(\mu)$ are distinct.\\
Let us calculate them. If $\lambda_1,\lambda_2$ are top two Lyapunov exponents of $\mu$\footnote{$\lambda_1=\lim_{n\rightarrow +\infty}{\frac{1}{n}\E (\log ||S_n||)}$ and $\lambda_1+\lambda_2=\lim_{n\rightarrow +\infty}{\frac{1}{n}\E (\log ||\bigwedge^2 S_n||)}$},
then the top Lyapunov exponent of $\eta_1(\mu)$ is $\lambda_1+\lambda_2$ and  the one corresponding to
$\eta_2(\mu)$ is clearly $\lambda_1$. So the problem occurs when $\lambda_2=0$.
This can happen for example when $\mu$ is a symmetric probability measure (i.e. the law of $X_1$ is the same as $X_1^{-1}$).\\
However, we can still find a probability measure whose support generates $\Gamma$ such that $\lambda_2\neq 0$, see Lemma \ref{ch2lyapunovcone}.

\section{Linearization  of algebraic varieties}
\label{ch2sublinear}
Let $\mathbf{G}$ be a semi-simple algebraic group defined on $\R$, $G$ its group of real points.\\

   The goal of this section is to  linearize every algebraic hypersurface of $\mathbf{G}$. More precisely,
  for every  proper algebraic hypersurface $\mathcal{V}$ defined over $\R$, we associate a finite dimensional rational real representation $(\rho,V)$ of $\mathbf{G}$, a linear form $L$ of
   $End(V)$ such that $\mathcal{V}=\{g\in \mathbf{G}; L\left(\rho(g) \right)=0\}$. In fact, we  will find  a representation $(\rho,V)$ of $\mathbf{G}$, a line $D$ in $V$, a    hyperplane
   $H$ in $V$ defined over $\R$  such that $\mathcal{V}=\{g\in \mathbf{G}; g\cdot D\subset H\}$ (see Lemma \ref{ch2lemma}).
   This has to be seen as a generalization of the well-known Chevalley theorem for subgroups (see Theorem \ref{ch2chevalley}).

  \begin{defi}[Matrix coefficients]
 If $(V,\rho)$ a finite dimensional representation of $G$, $ \langle \cdot, \cdot \rangle$ a scalar product on $V$,
   we call $\langle\rho(g)v,w\rangle$ for $v,w \in V$ a matrix coefficient and we denote by $C(\rho)$ the span of
the matrix coefficients of the representation $\rho$, thus a function $f\in C(\rho)$ can be
written $L \circ \rho$ where $L$ is a linear form on the vector space $End(V)$.\label{ch2matrix}\end{defi}

Let $\rho_1, \cdots , \rho_r$ be independent $\R$-rational irreducible representations of $\mathbf{G}$. Any $f_1 \in C(\rho_1),\cdots ,f_r\in C(\rho_r)$ are linearly independent provided that the representation
$\rho_i$ are pairwise non-isomorphic (see the proof of the Lemma \ref{ch2lemma} below). The set of elements of $G$ where
 such a linear dependance is realized defines clearly an algebraic hypersurface of $\mathbf{G}$. The following lemma
 says also that each algebraic hypersurface can be realized in this way.
\begin{lemme}\label{ch2lemma} For every algebraic hypersurface $\mathcal{V}$  of $\mathbf{G}$ defined over $\R$,
there exist a representation $(\rho,V)$ of $\mathbf{G}$, a line $D$ in $V$, a hyperplane $H$ of $V$ defined over $\R$ such that
$\mathcal{V}= \{g\in \mathbf{G}; g\cdot D \subset H\}$. In particular,  there exist a representation $(\rho,V)$ of $\mathbf{G}$
\textbf{whose irreducible sub-representations}, say $\rho_1,\cdots,\rho_r$, \textbf{occur only once},\; $f_1\in C(\rho_1),\cdots,f_r\in C(\rho_r)$ such that:
\begin{equation}\mathcal{V}(\R)=\{g\in G; \sum_{i=1}^r {f_i(g)}=0\}\label{ch2expression}\end{equation}
$\mathcal{V}$ is proper if and only if at least one of the $f_i$'s is non zero.\\
This is equivalent to say that there exists $A\in End(V_1)\oplus\cdots \oplus End(V_r)$ such that:
$$\mathcal{V}(\R)=\{g\in G;\; Tr\left(\rho(g)A\right)=0\}$$ with $\mathcal{V}$ proper if and only if there exists $i=1,\cdots,r$ such that the restriction of $A$ to $V_i$ is non zero. Here $Tr(M)$ denotes the trace of the endomorphism $M$. \end{lemme}


\begin{proof}[Proof]
Let $\R[\mathbf{G}]$ be the algebra of functions on $\mathbf{G}$, $\mathbf{G}$ acting on $\R[\mathbf{G}]$ by right translations:
 $g\cdot f(x)=f(x g)$ $\forall g,x \in \mathbf{G}$, $P$ the generator of the ideal vanishing on $\mathcal{V}$  (which is of rank one since
 $\mathcal{V}$ is a hypersurface). Then\; $g \in \mathcal{V} \Longleftrightarrow g \cdot P(1)=0$.\;
Consider the sub-representation $V=Vect(g\cdot P,g \in G)$. By \cite[Chapter 8, Proposition 8.6]{Humphreys},  $V$ is a finite dimensional  $\R$-rational
representation of $\mathbf{G}$. When $\mathcal{V}$ is proper, the subspace $H=\{f\in V; f(1)=0\}$ is a hyperplane defined over $\R$  so that $g \in \mathcal{V} \Longleftrightarrow g \cdot P\in H$
 and the first part of lemma is proved.
 $\mathbf{G}$ being semi-simple, we  decompose $(\rho,V)$ into irreducible sub-representations
 : $V=\oplus_{i=1}^rV_i$. Decomposing $P$ in the $V_i's$ gives easily (\ref{ch2expression}) with the only difference that the $V_i's$ are not necessarily pairwise non isomorphic.

Suppose for instance that $V_1\simeq V_2$. In this case,
there exists an invertible matrix $M$ such that $\rho_2(g)=M\rho_1(g)M^{-1}$ for every $g\in \mathbf{G}$.
Let $f_i=L_i \circ \rho_i$ where  $L_i$ is a suitable linear form on $End(V_i)$ for $i=1,2$. Then $f_2=\widetilde{L_2} \circ \rho_1$ where $\widetilde{L_2}$ is the  linear form
defined on $End(V_1)$ by  $\widetilde{L_2}(h)=L_2(M h M^{-1})$, $h\in End(V_1)$. Consequently, $f_2$ can be seen in
$C(\rho_1)$ so that
$f_1+f_2\in C(\rho_1)$ and $V_2$ can be dropped.  By updating $r$ if necessary,
we obtain (\ref{ch2expression}). At least one of the $f_i$'s is non zero, otherwise $\mathcal{V}$ would be $\mathbf{G}$.\\


$\bullet$ For the converse, we will show that if
 $\rho_1,\cdots,\rho_r$ are pairwise non isomorphic representations of $\mathbf{G}$, then
any $(f_1,\cdots ,f_r)\in C(\rho_1)\times\cdots\times C(\rho_r)$ are linearly independent.
A simple argument using the Peter-Weyl  theorem will immediately give the result for compact groups and a
unitary trick will allow us to conclude.\\
If $G$ were a compact Lie group, the proof would be a consequence of  Peter-Weyl theorem
for representations of compact groups (see for example \cite{weyl}): let $\sum$ be the collection of all irreducible representations of $G$
pairwise non isomorphic, $L^2(G)$ the set of all square integrable functions  with respect to the
Haar measure on $G$, then  $\{\sqrt{dim(\rho)}\rho_{i,j};\;\rho \in \sum;\; 1 \leq i,j\leq dim(\rho)\}$ forms an orthonormal basis of $L^2(G)$, where $\rho_{i,j}$
denotes the function on $G$ defined by $\rho_{i,j}(g)=\langle\rho(e_i),e_j\rangle$ for a certain basis $\{e_1,\cdots,e_{dim(\rho)}\}$ of the representation. We deduce immediately the linear  independence of any $f_1,\cdots,f_r$, where $f_i \in C(\rho_i)$ for each $i$. \\
Now we return to the general case.  If $\sum_{i=1}^r{ \lambda_if_i(g)}=0$ for all $g\in G=\mathbf{G}(\R)$ then by Zariski density, the same holds for all  $g\in \mathbf{G}(\mathbb{C})$. We decompose the $\rho_i$'s into $\mathbf{G}(\mathbb{C})$-irreducible representations. For sake of simplicity, we keep the notation $f_i$'s to denote the new matrix coefficients that follow from this decomposition.
 The Lie algebra $\mathfrak{g}$ of $\mathbf{G}(\mathbb{C})$ has a compact real form $\mathfrak{g}_0$
 (i.e. $\mathfrak{g}_0 \bigotimes_{R} \mathbb{C} = \mathfrak{g}$). To $\mathfrak{g}_0$  corresponds  a subgroup $G_0$ of
  $\mathbf{G}(\mathbb{C})$ which  is compact and Zariski sense in $\mathbf{G}(\mathbb{C})$. Hence an irreducible real representation of $\mathbf{G}(\mathbb{C})$ is $G_0$-irreducible. We conclude using the previous paragraph concerning Peter-Weyl theorem for compact groups.
\end{proof}

 \subsection{The particular case of subgroups}
Let $\mathbf{G}$ be  an algebraic group. The linearization of proper subgroups of $\mathbf{G}$ is Chevalley's theorem:

 \begin{theo}[Chevalley]\label{ch2chevalley} \cite{Humphreys} Let $\mathbf{H}$ be a proper subgroup of $\mathbf{G}$, then there exist a rational representation $(\rho,V)$ of $\mathbf{G}$,
   a line $D$ in $V$ such that $\mathbf{H}=\{g\in \mathbf{G}; g\cdot D=D \}$. \end{theo}

   In the particular case where the subgroup $\mathbf{H}$ is reductive, that is contains no proper connected unipotent subgroups, we have the
   following stronger statement:
   \begin{prop}\cite{borel} Let $\mathbf{H}$ be a proper reductive subgroup of $\mathbf{G}$, then there exists a rational real representation $(\rho,V)$ of $\mathbf{G}$, a non zero vector
   $x$ of $V$ such that $\mathbf{H}=\{g\in \mathbf{G}; g\cdot x=x \}$.\label{ch2reductive}\end{prop}
   The converse is true and is a theorem of Matsushima \cite{matsu} (see also \cite{invariant} for a recent proof).

\section{Preliminaries on algebraic groups}
\label{ch2subprel}

\subsection{The Cartan decomposition}
\label{ch2subcartan}
Let $\mathbf{G}$ be a semi-simple algebraic group defined over $\R$, $G$ its group of  real points,  $\mathbf{A}$ be a maximal $\R$-split torus of $\mathbf{G}$,
$\mathbf{X(A)}$ be the group of $\R$-rational characters of $\mathbf{A}$, $\Delta$ be the system of roots of $\mathbf{G}$ restricted to $\mathbf{A}$, $\Delta^+$ the system of positive roots (for a fixed order) and $\Pi$ the system of simple roots  (roots than cannot be obtained as product of two positive roots).\\
We consider the natural order on $\mathbf{X(A)}$: $\chi_1>\chi_2$ if and only if there exist non negative integers $\{n_\alpha; \alpha\in \Pi\}$ with  at least one non zero $n_\alpha$ such that $\frac{\chi_1}{\chi_2}=\prod_{\alpha\in \Pi} {\alpha^{n_\alpha}}$.\\
Finally  define $A^{\circ}=\{a\in A; \chi(a)\in ]0;+\infty[\;\forall \chi \in \mathbf{X(A)}\}$ and set
  $$A^{+}=\{a\in A^{\circ}\;;\;\alpha(a)\geq 1\;;\;\forall \alpha\in \Pi\}$$
Then there exists a compact $K$ of $G$ such that $$G=KA^+K\;\;\;\;\;\textrm{Cartan or $KAK$ decomposition}$$
(see \cite[Chapter 9, Theorem 1.1]{helg})

We denote by $\mathfrak{a}$ the Lie algebra of $\mathbf{A}$. The exponential map is a bijection between $A$ and $\mathfrak{a}$. A Weyl chamber is $\mathfrak{a}^+$.  We denote
 by $m$ the corresponding Cartan projection
$m: G \longrightarrow \mathfrak{a}^+$.

\subsection{Rational representations of algebraic groups}
A reference for this section is \cite{Humphreys} and \cite{Tits1}.
If $(\rho,V)$ is an $\R$-rational   representation of $\mathbf{G}$ then $\chi\in X(\mathbf{A})$ is called a weight of
$\rho$ if it is a common eigenvalue of $\mathbf{A}$ under $\rho$.
We denote  by $V_\chi$ the weight space associated to $\chi$ which  is
$V_\chi=\{x\in V; \rho(a)x=\chi(a)x\;\forall\;a\in \mathbf{A}\}$. The
following holds: $V=\oplus_{\chi\in X(\mathbf{A})}{V_\chi}$. Irreducible representations $\rho$ are characterized by
a particular weight $\chi_\rho$ called highest weight which has the following property:
 every weight $\chi$ of $\rho$ different from
$\chi_\rho$ is of the form $\chi=\frac{\chi_\rho}{\prod_{\alpha\in \Pi}{\alpha^{n_\alpha}}}$, where $n_\alpha \in \N$ for every
simple root $\alpha$.
The $V_{\chi}$'s are not necessarily of dimension $1$. When $\mathbf{G}$ is $\R$-split, $V_{\chi_\rho}$ is one dimensional.
 Recall that an element $\gamma\in GL_d(\R)$ is called proximal if it has a unique
eigenvalue of maximal modulus. A representation $\rho$ of a group $\Gamma$ is said to be proximal if the group $\rho(\Gamma)$
has a proximal element.  Thus, we obtain
\begin{lemme}\label{ch2split} Every $\R$-rational irreducible representation of an $\R$-split semi-simple algebraic
 group is proximal \end{lemme}

Let $\Theta_\rho=\{\alpha\in \Pi;\;\chi_\rho/\alpha \;\textrm{is a weight of $\rho$}\}$.

 \begin{prop}\cite{Tits1}
 For every $\alpha\in \Pi$, let $w_\alpha$ be the fundamental weight associated to $\alpha$.
 Then, there exists
 an $\R$-rational representation $(\rho_\alpha,V_\alpha)$ of $\mathbf{G}$  whose
 highest weight is a power of $w_\alpha$ and whose highest weight space $V_{w_\alpha}$  is one-dimensional.
 Moreover, $\Theta_{\rho_\alpha}=\{\alpha\}$ \label{ch2tits}\end{prop}

We record below a basic fact about root systems (\cite{bourbaki}).
\begin{prop} Every root $\alpha\in \Delta$ is of the form: $\alpha=\prod_{\beta\in \Pi}{w_\beta^{n_\beta}}$,
with $n_\beta\in \Z$, for every $\beta\in \Pi$.
\label{ch2funda}\end{prop}

\paragraph{Mostow theorem}\cite[\S 2.6]{Mostow1}
\label{ch2mostow}
Let $G=KAK$ be the Cartan decomposition of $G$, $(\rho,V)$ an irreducible rational real representation of $\mathbf{G}$.
There exists a scalar product on $V$ for which the elements of $\rho(K)$ are orthogonal and those of $\rho(A)$ are symmetric. In particular,
the weight spaces are orthogonal with respect to it.
The norm on $V$ induced by this scalar product is qualified by ``good''.

\subsection{Standard Parabolic subgroups and their representations}
\label{ch2subparabolic}
A reference for this section is \cite[\S 4]{gpesred}.\\
For every subset  $\theta\subset \Pi$, denote $\mathbf{A}_\theta=\{a\in \mathbf{A}; \alpha(a)=1 \forall \alpha \in \theta\}$ and let
$\mathbf{L}_\theta$ be its centralizer in $\mathbf{G}$.
 Denote by $\mathfrak{g}$  the Lie algebra of $\mathbf{G}$ and for every $\alpha\in \Delta$ denote by $\mathbf{U}_\alpha$  the unique closed unipotent subgroup
of $\mathbf{G}$ with Lie algebra: $\mathfrak{u}_{\alpha}=\mathfrak{g}_{\alpha}\oplus \mathfrak{g}_{2\alpha}$ where $\mathfrak{g}_{i\alpha} =\{X \in \mathfrak{g}; Ad(a)(X) = \alpha(a)^{i} X \;\forall a\in \mathbf{A}\}$.\\
Let $[\theta]\subset \Delta$ be the set of  roots which can be written as integral  combination of roots of $\theta$.
Denote by $\mathbf{U}_{\theta}$ the unipotent closed subgroup of $\mathbf{G}$ whose Lie algebra is
$$\mathfrak{u}_{\theta}=\bigoplus_{\alpha \in \Delta^+\setminus([\theta]\cap \Delta^+)}{\mathfrak{u}_\alpha}$$
We set $$\mathbf{P}_\theta=\mathbf{L}_\theta \mathbf{U}_\theta$$
This is the standard parabolic subgroup associated to $\theta$.
Its Lie algebra is

$$\mathfrak{p}_\theta=\mathfrak{z}\oplus  \bigoplus_{\alpha \in \Delta^+\cup[\theta]}{\mathfrak{u}_\alpha}$$
where $\mathfrak{z}$ is the Lie algebra of $\mathbf{Z}$, the centralizer of $\mathbf{A}$ in $\mathbf{G}$.
Notice that $\mathbf{P}_{\Pi}=\mathbf{G}$.\\

The following lemma will be useful to us for the proof of Theorem \ref{ch2tr3}.
\begin{lemme}  Let $(\rho,V)$ be a rational irreducible representation of $\mathbf{G}$ and consider $\theta \subset \Pi$. The  line generated by every non zero vector $x$ in the highest weight space of $V$ is fixed by $\mathbf{P}_\theta$ if  $\beta\not\in \Theta_\rho$ for every $\beta\in\theta$.
In particular, the line generated by any highest weight vector $x_\alpha$ of the representation $(\rho_\alpha,V_\alpha)$ defined in Proposition \ref{ch2tits} is fixed by
the standard parabolic $\mathbf{P}_\theta$ whenever $\alpha\not\in \theta$.
\label{ch2lemmepar}\end{lemme}

\begin{proof}[Proof] Let $\chi_\rho$ be the highest weight of $\rho$. We look at the action of the Lie algebra $\mathfrak{g}$ on $V$. It is clear
that $\mathfrak{g}_{-\beta}\cdot v \in V_{\chi_\rho-\beta}$ for every $v\in V_{\chi_\rho}$ and $\beta\in \Pi$. If $\beta\not\in \Theta_\rho$, then $\chi_\rho-\beta$ is not a weight of $\rho$ and hence $V_{\chi_\rho-\beta}=0$. The last part of the lemma is just recalling that the representation $\rho_\alpha$ defined in Proposition \ref{ch2tits} satisfies $\Theta_{\rho_\alpha}=\{\alpha\}$ \end{proof}

\section{Random matrix products - convergence in the Cartan decomposition}
\label{ch2subproba}



We will use in this section standard results in the theory of random matrix products. A nice reference is the book of Bougerol and La Croix  \cite{bougerol}.\\\\\

\subsection{Preliminaries}
\label{ch2subprelprob}
In the following, $G=\mathbf{G}(\R)$ is the group of real points of a semi-simple connected algebraic group, $\Gamma$ a Zariski dense subgroup of $G$, $\mu$ a  probability measure whose support generates $\Gamma$, $(\rho,V)$ an irreducible
$\R$-rational representation of $\mathbf{G}$ and $\chi_\rho$ its highest weight.
Let $\{X_n;n\in \N^*\}$ be independent random variables
 on $\Gamma$ with the same law $\mu$ and $S_n=X_n\cdots X_1$ the associated random walk .
Fix a Cartan decomposition of $G$ such that the section  $G\rightarrow KAK$ be measurable and denote for every $n\in \N^*$,
$S_n=K_nA_nU_n$ the corresponding  decomposition of $S_n$. If $\theta$ is a probability measure on $GL_d(\R)$, we denote by $G_\theta$ the smallest closed
subgroup containing the support of $\theta$. \\

  We consider the basis of weights of $V$ and
 the ``good norm'' given by Mostow theorem (Paragraph \ref{ch2mostow}).
 It induces a $K$-invariant norm on $\bigwedge^2 V$ and hence
  a $K$-invariant distance  $\delta(\cdot,\cdot)$ on the projective space $P(V)$,  called Fubini-Study distance,  defined by:
  $\delta([x],[y])=\frac{||x \wedge y ||}{||x|| ||y||}; [x],[y]\in P(V)$.\\
  We fix an orthonormal basis on each weight   space $V_\chi$, and for an element $g\in End(V)$, $g^t$ will be the transpose matrix of $g$
  in this basis.\\

 $G$ is isomorphic to a Zariski closed subgroup of $SL_d(\R)$ for some $d\in \N^*$  (see \cite{Humphreys}).
Let $i$ be such an isomorphism. We say way that $\mu$ has a moment of order one (resp.
an exponential moment) if for some (or equivalently any) norm on $End(\R^d)$, $\int{\log||i(g)||d\mu(g)}< \infty$   (resp. for some $\tau>0$,
 $\int{||i(g)||^\tau d\mu(g)}< \infty$). Lemma \ref{ch2exponential} below shows that is indeed a well defined notion,
 i.e. the existence of a moment of order
 one or an exponential moment is independent of the embedding.
 \begin{lemme}\label{ch2exponential} Let $G\subset SL(V)$ be the
     $\R$-points of a semi-simple algebraic group $\mathbf{G}$ and $\rho$ a finite dimensional $\R$-algebraic representation of $\mathbf{G}$. If
       $\mu$ has a moment of order one (resp. an exponential  moment)  then the image of $\mu$ under $\rho$ has also a moment of order one (resp. exponential moment).
       \end{lemme}
\begin{proof}[Proof]
Each matrix coefficient $(\rho(g))_{i,j}$ of $\rho(g)$, for $g\in G$, is a fixed polynomial in terms of the matrix coefficients of $g$. Since for the canonical norm, $||g||\geq 1$ for every $g\in G$, we see that there exists $C>0$ such that $||\rho(g)||\leq ||g||^C$ for every $g\in G$. This suffices to show the lemma. \end{proof}

Let us recall some definitions and well-known results.
\begin{defi}
A subgroup $\Gamma$ of $GL_d(\R)$ is called strongly irreducible if and only if the identity component of its Zariski closure does not fix a proper subspace. It is called proximal if it contains a proximal element (see Section \ref{ch2subprel}). \label{defdef}\end{defi}

The key result which prevents our results from being generalized to an arbitrary local field is Goldsheid-Margulis theorem we recall here

\begin{theo}\cite{Margulis} Let $d\geq 2$. A subgroup of $GL_d(\R)$ is strongly irreducible and proximal if and only if
 its Zariski closure is.  \label{ch2margulis}\end{theo}

\subsection{Geometry of the Lyapunov vector}
 First, let us recall the definition of the Lyapunov exponent.

\begin{dfprop}[Lyapunov exponent] If $\mu$ is a probability measure on $GL_d(\R)$, $||\cdot||$ a matricial norm on $End(V)$, $S_n=X_n \cdots X_1$ the corresponding random walk, then the Lyapunov exponent $L_{\mu}$  is $L_{\mu}=\lim \frac{1}{n} \E(\log ||S_n||)$ which exists by simple application of the subadditive lemma.\\
When $\mu$  have a moment of order one, the following a.s. limit holds $L_{\mu}=\lim \frac{1}{n} \log ||S_n||$. It can be  proved via the Kingman subadditive ergodic theorem \cite{Kingman}.  \label{ch2liapou}\end{dfprop}
A useful result will be the following
\begin{prop}\cite[Corollary 4 page 53]{bougerol}  Let $\theta$ be a probability measure on $GL_d(\R)$ with a moment of order one and such that $G_\theta:=\overline{\langle Supp(\theta) \rangle }$ is strongly irreducible. Then for every sequence $\{x_n;n\geq 0\}$ of vectors in $\R^d$
converging to some non zero vector $x\in \R^d$, $\frac{1}{n} \log ||S_n x_n|| \underset{n \rightarrow \infty}{\overset{\textrm{a.s.}}{\longrightarrow}} L_\theta$.       \label{ch2bougerol}\end{prop}

\begin{remarque} In \cite{bougerol}, the condition is made on the smallest closed sub-semi-group $\Gamma_\theta$ containing the support of $\theta$. There is no difference taking $\Gamma_\theta$ or $G_\theta$ because they have the same Zariski closure. Hence if one is strongly irreducible than the other satisfies the same property. This remark applies also for later applications when proximality is envolved (see for example the statement of Theorem \ref{ch2hausdorff}). This is due to Golsheild-Margulis theorem (Theorem \ref{ch2margulis}) which is special to the field of real numbers.\end{remarque}
\begin{dfprop}[Lyapunov vector]
 Suppose that $\mu$ has a moment of order one. Then the Lyapunov vector is the constant vector in the Weyl chamber $\mathfrak{a}^+$ of $G$ (see Section \ref{ch2subcartan}) defined as the following a.s. limit:
$$\frac{1}{n} m(S_n) \underset{n \rightarrow \infty}{\overset{\textrm{a.s}}{\longrightarrow}} Liap(\mu)$$
where $m$ is the Cartan projection (Section \ref{ch2subcartan}).
\label{ch2liapuvector}\end{dfprop}
\begin{proof}[Proof]
Let $\alpha\in \Pi$. Express $\alpha$ in terms of the fundamental weights (Proposition \ref{ch2funda}),
 $\alpha=\prod_{\beta \in \Pi}{w_\beta^{n_\beta}}$ where $n_\beta\in \Z$ for every $\beta \in \Pi$.
 For every $\beta\in \Pi$, consider the rational real irreducible representation $(\rho_\beta, V_\beta)$
   given by Proposition \ref{ch2tits} and a good norm on $V_\beta$ (Paragraph \ref{ch2mostow}).
   By the definition of $\rho_\beta$, there exists an integer $l_\beta$ such that for every $n\in \N^*$,
     $||\rho_\beta(S_n)||=w_\beta^{l_\beta}(A_n)$. Hence,
\begin{equation}\label{ch2nadet}\frac{1}{n}\log \;\alpha(A_n)=\sum_{\beta\in \Pi}{\frac{n_\beta}{l_\beta}\;
\frac{1}{n}\log \;||\rho_\beta(S_n)||}\end{equation}
By Definition/Proposition \ref{ch2liapou},  $\lim \frac{1}{n}\log \;\alpha(A_n) \overset{a.s.}{=} \sum_{\beta\in \Pi}{\frac{n_\beta}{l_\beta}\;
L_{\rho_{\beta}(\mu)}}$. Thus $Liap(\mu)$ is well defined.
\end{proof}

\begin{theo}\cite{Guivarch} Suppose that $\mu$ has a moment of order one. Then the Lyapunov vector $Liap(\mu)$ belongs to the interior of the Weyl chamber $\mathfrak{a}^+$, i.e. $\alpha\left(Liap(\mu) \right)>0$ $\forall \alpha\in \Pi$. \label{ch2guimo} \end{theo}
\begin{remarque} When the local field in not $\R$, the Lyapunov vector does not necessarily belong to the interior of
$\mathfrak{a}^+$. The reason is that Goldsheild-Margulis theorem (Theorem \ref{ch2margulis}) is valid only over the real field.  \end{remarque}
For the reader's convenience, we include a proof of Theorem \ref{ch2guimo}.
\begin{proof}[Proof] Without loss of generality, one can suppose $\Omega=G{^\N}=\{w=(w_i)_{i\in \N^*}; w_i\in G\}$, $\p$ the probability measure for which the coordinates $w_i$ are independent with law $\mu$ and $\mathcal{F}$ the $\sigma$-algebra generated by the coordinate maps $w_i$.\\
We want to show that for every $\alpha\in \Pi$, $l:=\lim\frac{1}{n} \log \;\alpha(A_n)>0$.
By equation (\ref{ch2nadet}), $l$ is the following constant: $l=\sum_{\beta\in \Pi}{\frac{n_\beta}{l_\beta}\; L_{\rho_\beta(\mu)} }$.
Let  $X=\prod_{\beta\in \Pi}{P(V_\beta)}$, $s$  the application on $G\times X$ defined by:
$$s\left(g, ([x_\beta])_{\beta\in \Pi_\alpha}\right) = \sum_{\beta\in \Pi_\alpha}{\frac{n_\beta}{l_\beta}\;
\log \;\frac{||\rho_\beta(g)x_\beta||}{||x_\beta||}}$$
It is immediate that $s$ is an additive cocycle on $G\times X$ for the natural action of $G$ on $X$.
Since $X$ is compact,  one can choose a $\mu$-invariant measure $\nu$ on $X$.\\
 Consider the dynamical system $E=\Omega\times X$, the distribution $\eta=\p \otimes \nu$ on $E$,
  the shift $\theta: E \rightarrow E$,  $\left((g_0,\cdots),x\right) \longmapsto \left((g_1,\cdots),g_0 \cdot x\right)$.
  Since $\nu$ is $\mu$-invariant, $\eta$ is $\theta$-invariant. We extend the definition domain of $s$ from $G\times X$ to $G^\N \times X$
  by setting
   $s(\omega,x):=s(g_0,x)$ if $\omega=(g_0,\cdots)$.
   Since $\mu$ has a moment of order one, Lemma \ref{ch2exponential} shows that the same holds for the image probability measure
$  \rho_\beta( \mu)$ for every $\beta\in \Pi$.
 Hence $s\in L_1(\eta)$. In consequence, we can apply the
  ergodic theorem (see \cite[Theorem 6.21]{brei}) which shows that $\frac{1}{n}{\sum_{i=0}^n {s\circ \theta^i (\omega,x)}}$
    converges for $\eta$-almost every $(\omega,x)$ to a random variable $Y$ whose expectation is
    $\iint{s(g,x)d\mu(g)d\nu(x)}$.  Since $s$ is a cocycle,\;
    $s\left(S_n(\omega),x\right)={\sum_{i=0}^n {s\circ \theta^i (\omega,x)}}$. Hence,
   $$ \lim_{n\rightarrow\infty}{\frac{1}{n}
s\left(S_n(\omega),x\right)} = Y\;\;\;\;\;;\;\;\;\; \E_\eta(Y)=\iint{s(g,x)d\mu(g)d\nu(x)}$$
But using Proposition \ref{ch2bougerol}, we see that a.s. $Y=l$ so that

   $$l=  \iint{s(g,x)d\mu(g)d\nu(x)}$$

By lemma \ref{ch2ergogui} below, $l$ is positive if   for $\eta$-almost every
$(\omega,x)$, $s(S_n(w),x) \underset{n\rightarrow\infty}{\longrightarrow} +\infty$.
Again by Proposition \ref{ch2bougerol}, for $\eta$-almost every $(w,x)$,
      $s(S_n(w),x)$ has the same behavior at infinity as the $\p$-almost everywhere behavior of
       $$\sum_{\beta\in \Pi_\alpha}{\frac{n_\beta}{l_\beta}\; \frac{1}{n}\log \;{||\rho_\beta(S_n)||}}=
       \log\;\alpha(A_n)$$
        In consequence, it suffices to show that
        $\alpha(A_n)\underset{n\rightarrow\infty}{\overset{\textrm{a.s.}}{\longrightarrow}}+\infty$.
        Indeed, the representation $\rho_\alpha$ is strongly irreducible because $G$ is Zariski connected. By Zariski density of $\Gamma$, the same holds for $\rho_\alpha(\Gamma)$. Moreover, by Goldsheild-Margulis Theorem (Theorem \ref{ch2margulis}) , $\rho_\alpha(\Gamma)$ is also proximal.
       By \cite[Theorem 3.1 page 50]{bougerol}, a.s. every limit point of $\frac{\rho_\alpha(S_n)}{||\rho_\alpha(S_n)||}$ is a rank one matrix.
     Hence, if $\rho_\alpha(A_n)=diag\left(a_1(n),\cdots,a_d(n)\right)$, then a.s.
    $a_2(n)/a_1(n)$ converges a.s. to zero. But $\Theta_{\rho_\alpha}=\{\alpha\}$ so that
    $\alpha(A_n)=a_1(n)/a_2(n)\;\underset{n\rightarrow\infty}{\longrightarrow}+\infty$. \end{proof}
\begin{lemme}\cite{dekk} Let $G$ be a group,  $X$ be a $G$-space,
$(X_n)_{n\in \N^*}$ a sequence of independent elements of $G$ with distribution $\mu$ and
$s$ an additive cocycle on $G\times X$. Suppose that $\nu$ is a $\mu$-invariant probability measure on $X$
such that:\\
(i) \; $\iint {s^+ (g,x) d\mu(g) d\nu(x)}< \infty$ \\
(ii)\; For $\p \otimes \nu$-almost every  $(w,x)$, \; $\lim_{n \rightarrow \infty} {s \left(X_n(w)\cdots X_1(w),x\right ) }=+ \infty $. \\
Then $s$ is in $L^1 (\p \otimes \nu)$ and $\iint{s(g,x) d\mu (g) d\nu(x)}>0$\label{ch2ergogui}\end{lemme}



The following lemma describes the geometry of the Lyapunov vector inside the Weyl chamber.

\begin{lemme} Let $\Gamma$ be a Zariski dense subgroup of $G$. Then for every finite union $F$ of hyperplanes
in  $\mathfrak{a}$ (see Section \ref{ch2subcartan} for the definition of $\mathfrak{a}$), there exist a probability measure $\mu$ on $\Gamma$ with an exponential moment whose support generates $\Gamma$ and whose Lyapunov vector $Liap(\mu)$ is not included in $F$.
In consequence, if $(V_1,\rho_1),\cdots, (V_r,\rho_r)$ are pairwise non isomorphic irreducible representations of $\mathbf{G}$, then one can exhibit a probability
measure $\mu$ whose support generates $\Gamma$, a permutation $\sigma$ of $\{1,\cdots, r\}$ such that $L_{\rho_{\sigma(1)}(\mu)}>\cdots>L_{\rho_{\sigma(r)}(\mu)}$ (See Definition \ref{ch2liapou}).
\label{ch2lyapunovcone}\end{lemme}
\begin{proof}[Proof]  We recall the definition of the Jordan projection. Every element $g\in G$ has a decomposition: $g=g_e g_h g_u$ with $g_e$ elliptic (i.e. included in a compact subgroup), $g_h$ hyperbolic (i.e. conjugated
to an element $a(g)$ in $A^+$) and $g_u$ unipotent commuting with $g_h$. The Jordan projection $j: G \longrightarrow \mathfrak{a}^+$ is defined by
$\lambda(g)=\log a(g)$. \\
Y. Benoist proved in \cite{cone1} that
the smallest cone ${l}_\Gamma$ in $\mathfrak{a}^+$ containing $j(\Gamma)$ has a non empty interior. Moreover, he showed in
\cite{cone2} that $j(\Gamma)$ fills completely ${l}_\Gamma$ in the sense that every open cone in ${l}_\Gamma$
contains an infinite elements of $j(\Gamma)$. We deduce that $j(\Gamma)$ cannot be supported on any finite union of hyperplanes in
$\mathfrak{a}$. \\
Let now $F$ be such a finite union of hyperplanes,  $g\in \Gamma$ such that $j(g)\not\in F$. The spectral radius formula shows that $\frac{1}{n} m(g^n) \underset{n \rightarrow \infty}{\longrightarrow }
 j(g)\not\in F$ where $m$ is the Cartan projection (Section \ref{ch2subprel}). This is equivalent to say that the Dirac probability measure $\mu=\delta_g$ supported on $\{g\}$ satisfies $Liap(\mu)\not\in F$.\\
Let us perturb $\mu$ on $\Gamma$, that this  define a sequence of probability measure $\mu_n$ with an exponential moment whose support generates $\Gamma$   such that $\mu_n$ converge weakly
to $\mu$ , for example $\mu_n=(1-1/n)\mu +\eta/n$ where $\eta$ is a probability measure with an exponential moment whose support generates  $\Gamma$. It is easy to see (see for example \cite[Corollary 7.3, page 72-73]{bougerol}) that the Lyapunov vector depends continuously
 on the probability measure
so that $Liap(\mu_n)$ converge to $Liap(\mu)$. Hence, for $n$ big enough,
$\mu_n$ is a probability measure on $\Gamma$ with $Liap(\mu_n)\not\in F$.\\

Now we prove the last part of the lemma. Let $\rho_1,\cdots,\rho_r$ be $r$ rational real irreducible representations of $\mathbf{G}$ and denote by $\chi_{\rho_i}$ the highest weight of $\rho_i$. Recall that the set $\Pi$ of simple roots is a basis of the space $X(A)$  of the rational characters of $A$. Hence for every $i=1,\cdots r$, there exist real numbers $\{n_{i,\alpha};\alpha\in \Pi\}$ with at least one non zero number such that:
 $$\log{\chi_{\rho_i}}=\sum_{\alpha\in \Pi}{n_{i,\alpha}{\log\alpha}}$$
For every $i<j$, denote by $H_{i,j}$ the following hyperplane of  $\mathfrak{a}$:
$$H_{i,j}=\{x\in \mathfrak{a}; \sum_{\alpha\in \Pi}{n_{i,\alpha}\log{\alpha(x)}}=\sum_{\alpha\in \Pi}{n_{j,\alpha}\log{\alpha(x)}}\}$$
Set $F=\cup_{i<j}{H_{i,j}}$. Applying the first of the lemma shows that there exists a probability measure on $\Gamma$ with an exponential moment such that $Liap(\mu)\not\in F$. This ends the proof because
 for every $i=1,\cdots,r$, $$L_{\rho_i(\mu)}=\lim \frac{1}{n} \log{ \chi_{\rho_i}(A_n)}$$
 \end{proof}

\subsection{Estimates in the $A$-part}
The following theorem gives an estimates in the $A$-part of the Cartan decomposition of the random walk.
It can be proved by the same techniques of \cite{aoun}
where the theory of random matrix products is treated over an arbitrary local field.
However, since we are working here in $\R$, we will use another route and apply the large deviation theorem of Le Page \cite{Page} in $GL_d(\R)$
we recall below. First, let us state our result:
\begin{theo}\label{ch2ratioA} [Ratio in the $A$-component] Suppose that $\mu$ has an exponential moment
 then for every $\epsilon>0$ and every non zero weight $\chi$ of $\rho$ distinct from $\chi_\rho$,
\begin{equation}\label{ch2mamisallile}\displaystyle \limsup_{n\rightarrow \infty} \big[\E [( \frac{\chi(A_n)}{\chi_{\rho}(A_n)})^{\epsilon} ]\big]^{\frac{1}{n}}< 1\end{equation}
Moreover, if $\rho_1$, $\rho_2$ are two irreducible rational real representations of $\mathbf{G}$ such that $L_{\rho_1(\mu)}>L_{\rho_2(\mu)}$ (Definition \ref{ch2liapou}), then for every $\epsilon>0$:
\begin{equation}
\displaystyle \limsup_{n\rightarrow \infty} \big[\E [( \frac{\chi_{\rho_2}(A_n)}{\chi_{\rho_1}(A_n)})^{\epsilon} ] \big]^{\frac{1}{n}}< 1\label{ch2lkle}\end{equation}
\end{theo}
Before giving the proof, we recall Le Page large deviation theorem in $GL_d(\R)$:
\begin{theo}\cite{Page}[Large deviations in $GL_d(\R)$]

Let $\mu$ be a probability on $GL_d(\R)$ having an exponential moment and such
 that $G_\mu$ is strongly irreducible. Let  $S_n=X_n\cdots X_1$ be the corresponding random walk. Then for every $\epsilon>0$,
$$\displaystyle \limsup_{n\rightarrow \infty} \big[\p \left(\big|\frac{1}{n} \log ||S_n|| - L_\mu \big| >\epsilon  \right) \big]^{\frac{1}{n}}< 1$$
\label{ch2page}
A similar estimate holds for $\frac{1}{n}\log ||S_n x ||$ for every non zero vector $x\in \R^d$.
\end{theo}

\begin{proof}[Proof of Theorem \ref{ch2ratioA}]


 For every $\beta\in \Pi$, a similar large deviation inequality as in Theorem \ref{ch2page} holds for the quantity $\frac{1}{n} \log ||\rho_\beta(S_n)||$ because $\rho_\beta$ is strongly irreducible and $\rho_\beta(\mu)$ has an exponential moment by
Lemma \ref{ch2exponential}.  Hence by equation (\ref{ch2nadet})  a large deviation inequality holds for $\frac{1}{n} \log \alpha(A_n)$ for every $\alpha\in \Theta$.
Since $\chi_\rho/\chi= \prod_{\alpha\in \Pi}{\alpha^{n_\alpha}}$ for non-negative integers $\{n_\alpha; \alpha\in \Pi\}$,
  we get for $\lambda= - \sum_{\alpha\in \Pi}{n_\alpha\;\lim_{n\rightarrow\infty}\frac{1}{n} \log\;\alpha(A_n)}$ and for every $\epsilon>0$,
 \begin{equation}\label{ch2ma3ash}\displaystyle \limsup_{n\rightarrow \infty}  \big[\p \left( \big|\frac{1}{n} \log\; \frac{\chi(A_n)}{\chi_\rho(A_n)} - \lambda \big| >\epsilon \right) \big]^{\frac{1}{n}}< 1\end{equation}

By Theorem \ref{ch2guimo}, $\lambda<0$. Hence, by  relation  (\ref{ch2ma3ash}), there exists $\rho_1,\rho_2\in ]0,1[$ such that for all large $n$:
$\p \left( \frac{\chi(A_n)}{\chi_\rho(A_n)} \geq  \rho_1^n \right) \leq \rho_2^n $.
Since $\chi(a)\leq \chi_\rho(a)$ for every $a\in A^+$, we get for every $\epsilon>0$,
 $\E \Big[\left( \frac{\chi(A_n)}{\chi_\rho(A_n)} \right)^\epsilon\Big] \leq \rho_1^{\epsilon n}+\rho_2^n$. This shows (\ref{ch2mamisallile}).\\
By the same  large deviation techniques, one can show (\ref{ch2lkle}).\end{proof}

\subsection{Estimates in the $K$-parts}
Recall that we fix  a measurable section of the Cartan decomposition $G \rightarrow KAK$ and the corresponding decomposition of  the random walk $S_n$ is denoted by $S_n=K_nA_nU_n$.
Our next task is to prove the following theorem which gives the convergence in the $K$-parts of the Cartan decomposition
of the random walk.\\
This result was proved in our previous work \cite[Theorem 4.33]{aoun} over an arbitrary local field. We give here another proof special to archimedean fields.
\begin{theo}\label{ch2conv}[Exponential convergence of the $K$-components]
Suppose that $\mu$ has an exponential moment and
$\rho$ is proximal. Let $v_\rho$ be a highest weight vector. Then there exists a random variable $Z$ on the projective space $P(V)$ such that for every $\epsilon>0$:

$$\displaystyle \limsup_{n\rightarrow \infty}\big[\E\left( {\delta (U_n^{-1}\cdot [v_\rho],Z) }^{\epsilon} \right)\big]^{\frac{1}{n}} < 1$$
 Here, for $M\in GL(V)$, we have denoted by $M^t$ the transpose matrix of $M$ with respect to the basis of weights. We recall that $\delta$ is the Fubini-Study distance (see the beginning of Section \ref{ch2subprelprob}).
A similar estimate holds if we replace $U_n$ with $k(X_1\cdots X_n)$ where $k(g)$ is the $K$-component of $g\in G$ for
the fixed $KAK$ decomposition in $G$.
\end{theo}

\begin{proof}[Proof] We recall that by Mostow theorem, there exists a scalar product $\langle\cdot\rangle$ on $V$ such that the weight spaces are orthogonal and $K$ acts by isometries and that we choose an orthonormal  basis in each weight space so that $\rho(K)\rho(K)^t$ is the trivial group. \\
 For every $n\in \N^*$, every non zero weight $\chi$, we denote by $Q_\chi(n)$ the orthogonal projection on the space
 $U_n^{-1}\cdot V_\chi= \rho(U_n)^tV_\chi$. In particular $Q_{\chi_\rho}(n)$ is the projection on the line $\R U_n^{-1}\cdot v_{\rho}$, where
  $v_\rho$ is a highest weight vector (it is one-dimensional because $\rho$ is proximal).
  We will show that for every $\epsilon>0$ small enough:
\begin{equation}
\displaystyle \limsup_{n\rightarrow \infty}\big[\E(||Q_{\chi_\rho}(n) - Q_{\chi_{\rho}}(n+1)||^{\epsilon})\big]^{\frac{1}{n}}< 1\label{ch2yara}\end{equation}
This ends the proof because if $x$ and $y$ are two non zero vectors of $V$
and $Q_x$ and $Q_y$ are the orthogonal projections on the lines $\R x$ and $\R y$,
then $||Q_x - Q_y||\geq \frac{1}{2}\delta([x],[y])$, so that (\ref{ch2yara}) would imply by the Markov property
that
 $\{U_n^{-1}\cdot [v_\rho];\;n\geq 0\}$ is a.s. a Cauchy series in the projective
 space $P(V)$. Hence it converges to some variable $Z$. By Fatou lemma and the triangular inequality, we get for some $t=t(\epsilon)\in ]0,1[$ and all large $n$:
$\E\left({\delta (Z,U_{n}^{-1}\cdot [v_\rho])}^\epsilon\right)\leq \liminf_{m \rightarrow \infty}\;\E\left({\delta (U_{m}^{-1}\cdot [v_\rho],U_{n}^{-1}\cdot [v_\rho])}^\epsilon\right)\leq {t^n}$. \\
Now we prove (\ref{ch2yara}). For every $n\in \N^*$, $\sum_{\chi}{Q_\chi}$ is the identity operator,
where the sum is over all the non zero weights of $(\rho,V)$. Moreover, two orthogonal projections commute, hence:
 $||Q_{\chi_\rho}(n) - Q_{\chi_\rho}(n+1)||\leq \sum_{\chi\neq \chi_\rho} { ||Q_{\chi_\rho}(n+1)Q_{\chi}(n)||+ ||Q_{\chi_\rho}(n)Q_{\chi}(n+1)|| }$. \\
Fix a weight $\chi\neq \chi_\rho$. First we show that  $\E(||Q_{\chi_\rho}(n+1)Q_{\chi}(n)||^\epsilon)$ is sub-exponential for every $\epsilon>0$ small enough. This is equivalent to prove that there exists $\eta\in ]0,1[$ such that for all large $n$:

$$\E \left( \big[\displaystyle{\sup_{x \in U_n^{-1}\cdot V_\chi; ||x||=1}{|Q_{\chi_\rho}(n+1) (x)|}}\big]^\epsilon \right) \leq \eta^n $$
Let $x\in U_n^{-1}\cdot V_\chi$ of norm one and $y_n=Q_{\chi_\rho}(n+1)(x)$, i.e. the orthogonal projection of
 $x$ on the line $U_{n+1}^{-1}\cdot V_{\chi_\rho}$.  Now we evaluate
 $||S_{n+1}\cdot x||$ in two different ways. On the one hand,
  \begin{equation}||S_{n+1}\cdot x||=||X_{n+1}S_n\cdot x||\leq ||\rho(X_{n+1})|| \;||S_n\cdot x||=||\rho(X_{n+1})||\;\chi(A_n)\label{ch21}\end{equation}
On the other hand,  $\langle S_{n+1}\cdot (x-y_n),S_{n+1}\cdot y_n\rangle=\langle(x-y_n),S_{n+1}^tS_{n+1}\cdot y_n\rangle=0$ because
$x-y_n \perp U_{n+1}^{-1}\cdot V_{\chi_\rho}$ and if $y_n=U_{n+1}^{-1}\cdot  z_n$ for some $z_n\in V_{\chi_\rho}$, then
$S_{n+1}^tS_{n+1}\cdot y_n = U_{n+1}^{-1}A_{n+1}^2 U_{n+1}U_{n+1}^{-1}\cdot z_n={\chi_\rho} ^2(A_{n+1}) y_n\in U_{n+1}^{-1}\cdot V_{\chi_\rho}$.
Hence \begin{equation}\label{ch22}||S_{n+1}\cdot x||=\sqrt{||S_{n+1}\cdot  y_n||^2+||S_{n+1}\cdot  (x-y_n)||^2} \geq
||S_{n+1}\cdot y_n||={\chi_\rho}(A_{n+1})\;||y_n||\end{equation}
Combining (\ref{ch21}) and (\ref{ch22}) gives:
$$\displaystyle{\sup_{x \in U_n^{-1}\cdot V_\chi; ||x||=1}{||Q_{\chi_\rho}(n+1) (x)||}} =||y_n|| \leq ||\rho(X_{n+1})|| \;\frac{\chi(A_{n})}{\chi_\rho(A_{n+1})}$$
But for every $p\in \N^*$, $||\rho(S_p)||=\chi(A_p)$ (because the norm on $V$ is $K$-invariant). Hence,
\begin{equation}\label{ch2badama3}\displaystyle{\sup_{x \in U_n^{-1}\cdot V_\chi; ||x||=1}{||Q_{\chi_\rho}(n+1) (x)||}} =||y_n|| \leq ||\rho(X_{n+1})||\cdot  ||\rho(X_{n+1}^{-1})|| \;\frac{\chi(A_{n})}{\chi_\rho(A_{n})}\leq ||\rho(X_{n+1})||^d \;\frac{\chi(A_{n})}{\chi_\rho(A_{n})} \end{equation}
Last inequality is due to the relation $||g^{-1}||\leq ||g||^{d-1}$ true for every $g\in SL_d(k)$.
By Lemma \ref{ch2exponential}, the probability measure $\rho ( \mu)$ has an exponential moment so that
there exists $C\geq 1$ such that for all $\epsilon>0$ small enough  $\E (||\rho(X_{n+1})||^\epsilon)< C$. By Theorem \ref{ch2ratioA},
 for every $\epsilon>0$ small enough, some $\eta(\epsilon)\in ]0,1[$ and all $n$ large enough:
$\E \Big[\left(\frac{\chi(A_{n})}{\chi_\rho(A_{n})}\right)^\epsilon\Big] \leq \eta(\epsilon)^n$.
It suffices to apply Cauchy-Schwartz inequality to (\ref{ch2badama3}) to obtain the sub-exponential behavior of
$\E \left( ||Q_{\chi_\rho}(n)Q_{\chi}(n+1)||^\epsilon \right)$.

To bound $\E(||Q_{\chi_\rho}(n)Q_{\chi}(n+1)||^\epsilon)$ we apply the same reasoning as above: we fix $x\in V_{\chi}(n+1)$ of norm one
 and denote by $y_n$ its projection on $U_n^{-1}\cdot V_{\chi_\rho}$. Then, we evaluate $||S_n\cdot x||$ in two ways:
$$||S_n\cdot x||=||X_{n+1}^{-1}S_{n+1}\cdot x||\leq {||\rho(X_{n+1}^{-1})||}\;\chi (A_{n+1})$$
$$||S_n\cdot x||=\sqrt{||S_n\cdot (x-y_n)||^2 + ||S_n \cdot y_n||^2}\geq ||S_n\cdot y_n||= ||y_n||\chi_{\rho}(A_n)$$
The end of the proof is the same as above. For the law of $Z$, see the following remark.
\end{proof}
\begin{remarque} \label{ch2conkak}[Identification of the limit] By the Markov inequality and the Borel-Cantelli lemma,
Theorem \ref{ch2conv} shows that $U_n^{-1}[v_\rho]$ converges towards some
random variable $Z$.
In fact, the law of $Z$ is the unique $\rho(\mu)^t$-invariant probability measure on $P(V)$ (see for example \cite[Proposition 3.2 page 50]{bougerol}).
 \end{remarque}

We will also use the following lemma:
\begin{lemme} Let $\mu$ be a probability measure on $GL_d(\R)$ with an exponential moment, such that the smallest closed group $G_\mu$ containing the support is $\mu$ is strongly irreducible and proximal, then
 \begin{equation}\limsup_{n\rightarrow +\infty} \frac{1}{n} \log {\p \left(S_n [x] \in H \right)}< 0\label{decr}\end{equation}
 uniformly on $x\in \R^d \setminus \{0\}$ and the hyperplanes $H$ of $\R^d$. \label{ch2hyperplane}\end{lemme}
 \begin{remarque}In \cite[Theorem 4.18]{aoun}, we have proved the previous lemma over an arbitrary local field.
 We will give here a short proof since we are working in the field of real numbers.
\end{remarque}
\begin{proof} With the assumptions of the lemma,  $S_n[x]$ converges in law towards a random variable $Z$
with law the unique $\mu$-invariant probability measure $\nu$ on the projective space $P(\R^d)$. Moreover,
the convergence is with exponential speed in the following sense (see \cite[Chapter V, Theorem 2.5]{bougerol}):
 there exists $\alpha>0$ such that for every $\alpha$-holderian function on $P(\R^d)$,

 $$\Big|\E \Big[(f\left(\rho(S_n)[x] \right) \Big) - \int{f d\nu} \Big|\leq ||f||_\alpha \rho^n$$
 where $$||f||_\alpha=\underset{[x]\neq [y]\in P(\R^d)}{\sup}{|f([x])-f([y])|}{\delta^\alpha([x],[y])}$$
 and $\delta(\cdot, \cdot)$ is the Fubiny-Study distance on the projective space $P(\R^6)$.
 But the limiting measure $\nu$ has some regularity, its Hausdorff dimension is positive and satisfies:

$$\underset{\textrm{$H$ hyperplanes in $R^6$}}{\sup} \nu \left(\{[x]\in P(\R^6); \delta([x],H)\leq \epsilon\} \right)\leq C \epsilon^\alpha$$
for some $C,\alpha>0$ (see \cite[Chapter VI, Corollary 4.2]{bougerol}).
We can now easily conclude.\end{proof}

Finally, we quote a useful result  from \cite{aoun}.
\begin{theo}\cite[Theorem 4.35]{aoun}[Asymptotic independence of the $K$-components] With the same assumptions as in Theorem \ref{ch2conv},
there exist
 \textbf{independent random variables} $Z$ and $T$ with respective laws the unique $\rho(\mu)^t$ (resp. $\rho(\mu)$)-
 invariant probability measure on $P(V)$  such that for every $\epsilon>0$,
    every $\epsilon$-holder (real) function $\phi$ on $P(V)\times P(V)$ and all large $n$ we have:
$$\big|\E\left(\phi([U_n^{-1}\cdot v_{\rho}],[K_n\cdot v_\rho]) \right) - \E \left( \phi(Z,T) \right) \big|\leq ||\phi||_{\epsilon} \rho(\epsilon)^n $$
where \;\;\;\;\;$||\phi||_\epsilon=
   \underset{{ [x],[y],[x'],[y']}}{Sup} \;{\frac{\big|\phi([x],[x'])-\phi([y],[y'])\big|}{\delta
 ([x],[y])^{\epsilon}+\delta([x'],[y'])^{\epsilon}}}$.
\label{ch2independence}\end{theo}

\section{Proof of the main theorems}
\label{ch2subproof1}

The proof of the main theorems we presented in the introduction is based on the following

\begin{theo} Let $\mathbf{G}$ be a semi-simple algebraic group defined over $\R$, $G$ its group of real points, let  $(\rho,V)$ be a rational real
representation of $\mathbf{G}$ such that its irreducible sub-representations $(\rho_1,V_1),\cdots,(\rho_r,V_r)$ are pairwise non isomorphic and let finally
 $A\in End(V_1)\oplus\cdots\oplus End(V_r)$ such that its projection on $End(V_1)$ is non zero. Consider a probability measure $\mu$ on $G$ with an exponential moment
 and such that $G_\mu:=\overline{\langle Supp(\mu)\rangle}$ is
  Zariski dense in $G$. Denote by $\{S_n; n\geq 0\}$ the corresponding random walk. Assume that :

    \begin{enumerate}
  \item $\rho_1$ is proximal.
  \item $L_{\rho_1(\mu)}> L_{\rho_i(\mu)}$ , $i=2,\cdots, r$  (see Definition \ref{ch2liapou}).
  \end{enumerate}

Then for every $\epsilon>0$ there exists $\rho(\epsilon)\in ]0,1[$ such that for all large $n$:
$$\p \Big(\big| \frac{1}{n} \log|Tr\left(\rho(S_n)A\right)| - L_{\rho_1(\mu)}\big| >\epsilon \Big) \leq \rho(\epsilon)^n$$
In particular, $Tr\left(\rho(S_n)A\right)$ vanishes only with  a probability decreasing exponentially fast to zero, and $\frac{1}{n} \log \Big|Tr\left(\rho(S_n) A\right)\Big|$
 converges a.s. towards $L_{\rho_1(\mu)}$.
\label{ch2theo1}\end{theo}



Assumption 1 in Theorem \ref{ch2theo1} is fulfilled whenever $\mathbf{G}$ is $\R$-split (see Lemma \ref{ch2split}). We provide two sufficient conditions for assumption 2 to hold: a probabilistic one and a determinist (algebraic) one.

\begin{remarque}[A probabilistic sufficient conditions for assumption $2$]

 Lemma \ref{ch2lyapunovcone} proves  that assumption 2 is fulfilled whenever the Lyapunov vector $Liap(\mu)$ does not belong to a finite union of hyperplanes in the Weyl chamber $\mathfrak{a}^+$ . \end{remarque}


\begin{remarque}[An algebraic sufficient conditions for assumption $2$]

Let $\chi_i$ be the highest weight of $V_i$, $i=1,\cdots,r$.
A necessary  condition for $2$ to hold is
that $\chi_1/\chi_i=\prod_{\alpha \in\Pi}{\alpha^{n_\alpha}}$ for some non negative integers $\{n_\alpha; \alpha\in \Pi\}$ with at least one non zero $n_\alpha$.
This is easily checked using the fact that the Lyapunov vector is in the interior of the weyl chamber (Theorem \ref{ch2guimo}).\\
See the applications of this remark in the proof of Theorem \ref{ch2tr2}
 \label{ch2nec}\end{remarque}

\begin{proof}[Proof]
Without loss of generality, we can assume $r=2$. Let $d=dim(V)$, $p=dim(V_1) $, $B_1=(v_1,\cdots,v_p)$ (resp. $B_2=(v_{p+1},\cdots, v_d)$)  a basis of $V_1$ (resp. $V_2$) consisting of weight vectors. We impose $v_1$ to be a highest weight.
This gives a basis $B=(B_1,B_2)$ of $V$. The scalar products on $V_1$ and $V_2$ given by Theorem \ref{ch2mostow}
induce naturally a scalar product on $V$ for which $V_1$ and $V_2$
are orthogonal.
 In the basis $B$,
   $\rho(A_n)=diag(\rho_1(A_n),\rho_2(A_n))=diag(a_1(n),\cdots,a_d(n))$ with $a_1(n)=\chi_{\rho_1}(A_n)$ and $a_{p+1}(n)=\chi_{\rho_2}(A_n)$ (notations of Section \ref{ch2subprel}).
Let $W_{\rho_i}$ be the set of non zero weights of $(V_i,\rho_i)$, $i=1,2$.   A simple computation gives:

\begin{eqnarray}Tr(\rho(S_n)A) &=& Tr(\rho(K_n)\rho(A_n)\rho(U_n)A) = Tr(\rho(A_n)\rho(U_n)A \rho(K_n))\nonumber\\
&=&\sum_{i=1}^d {a_i(n) \langle \rho(K_n) v_i,A^t\rho(U_n)^{t}v_i\rangle}\nonumber\end{eqnarray}
where $S_n=K_nA_nU_n$ is the Cartan decomposition of $S_n$ (see Section \ref{ch2subcartan}).
Since $\rho_1$ is proximal, $a_2(n)=\chi(A_n)$
for some weight $\chi\in W_{\rho_1}$ distinct from $\chi_\rho$. Then,
$$Tr(\rho(S_n)A)=  \chi_{\rho_1}(A_n) \Big[ \langle K_n\cdot v_{\rho_1},A^tU_n^{-1}\cdot v_{\rho_1}\rangle + \sum_{\chi\neq \chi_{\rho_1}
\in W_{\rho_1}}\;{O\left(\frac{\chi(A_n)}{\chi_{\rho_1}(A_n)} \right)}+ \sum_{\chi \in {W_{\rho_2}}}\;
{O\left(\frac{\chi(A_n)}{\chi_{\rho_1}(A_n)} \right)}\Big]$$

Le Page large deviations theorem (Theorem \ref{ch2page}) shows that for every $\epsilon>0$ and some $\rho\in ]0,1[$:
$$\p \left(exp(nL_{\rho_1(\mu)}-n\epsilon)\leq\chi_{\rho_1}(A_n)\leq exp(nL_{\rho_1(\mu)}+n\epsilon)\right) \geq 1-\rho^n$$

Next we show that for every $\chi\neq \chi_{\rho_1} \in W_{\rho_1}$ and
$\chi\in W_{\rho_2}$ and every $\epsilon>0$:

$$\displaystyle \limsup_{n\rightarrow \infty}\big[\E \left(\frac{\chi(A_n)}{\chi_\rho(A_n)} \right)^\epsilon \big]^{\frac{1}{n}}< 1$$
Indeed, for $\chi\neq \chi_{\rho_1} \in W_{\rho_1}$, this follows from Theorem \ref{ch2ratioA} and the fact that $\rho_1$ is proximal. For $\chi \in W_{\rho_2}$, this follows also from Theorem \ref{ch2ratioA} and assumption $2$.

Hence, by the Markov property,  there exist $\epsilon_1,\epsilon_2\in ]0,1[$ such that for all $n$ large enough: $\p \left(\frac{\chi(A_n)}{\chi_\rho(A_n)} \geq \epsilon_1^n \right) \leq \epsilon_2^n $.
 The following proposition applied to the (non trivial)
projection of $A$ on $V_1$ and to the representation $(\rho_1,V_1)$ ends the proof.
\end{proof}

\begin{prop}
Let $\mathbf{G}$ be a semi-simple algebraic group defined over $\R$, $G$ its group of real points, $\Gamma$ a Zariski dense subgroup of $G$, $(\rho,V)$ an irreducible rational real representation of $\mathbf{G}$,
 $\mu$ a probability measure with an exponential moment and whose support generates $\Gamma$.
 If $\rho$ is proximal, then for any non zero endomorphism $A\in End(V)$:
$$\displaystyle \limsup_{n\rightarrow \infty}\big[\p \left(|\langle K_n\cdot v_{\rho},AU_n^{-1}\cdot v_{\rho}\rangle|\leq t^n \right) \big]^{\frac{1}{n}} < 1$$
where $v_{\rho}$ is a highest weight vector.
\label{ch2propp}\end{prop}
Before giving the proof, we recall the following remarkable theorem of Guivarc'h:
\begin{theo}\label{ch2hausdorff}\cite{Guivarch3} Let $\mu$ be a probability measure on $GL_d(\R)$ having an exponential moment
and such that $G_\mu$ is strongly irreducible and proximal. Denote by $\nu$ the unique $\mu$-invariant probability measure
on the projective space $P(\R^d)$ . Then there exists $\alpha>0$ (small enough) such that:
$$Sup\{\int{\frac{1}{|\langle\frac{x}{||x||},\frac{y}{||y||}\rangle|^\alpha}d\nu([x])\;\;;\;\;y\in \R^d \setminus\{0\}}\}< \infty $$
In particular, if $Z$ is a random variable with law $\nu$,  there exists a constant $C>0$ such that:
$$Sup\{\p (|\langle Z,\frac{x}{||x||}\rangle|\leq \epsilon );\;\; x\in \R^d\setminus\{0\}\}\leq C\epsilon^\alpha$$\end{theo}
\begin{proof}[Proof of Proposition \ref{ch2propp}]
\begin{itemize}
  \item  Let $\eta$ the function defined on $P(V)\times P(V) \rightarrow \R$
by $\eta([x],[y]) =|\langle x,Ay\rangle|$ where $x$ and $y$ are two representative of $[x]$ and $[y]$
in the sphere of radius one. The function $\eta$ is lipshitz with lipshitz constant $\leq Max\{1,||A||\}$.
  \item For every $a>0$, let
 $\psi_a$ be  the  function defined on $\R$ by $\psi_a(x)= 1$ if $x\in[-a;a]$;
 affine on $[-2a;-a[ \cup ]a,2a]$ and zero otherwise.
One can easily verify that $\psi_a$ is lipshitz with constant equal to $\frac{1}{a}$.\\
Note also that \begin{equation}\mathds{1}_{[-a,a]} \leq \psi_a \leq \mathds{1}_{[-2a,2a]}\label{ch2yy}\end{equation}
\end{itemize}
Define for $a>0$, $\phi_a=\psi_a \circ \eta$.
 By the previous remarks,   $\phi_a$ is lipshitz  with lipshitz constant: $||\phi_a|| \leq \frac{Max\{1,||A||\}}{a}$.\\
By Theorem \ref{ch2independence} there exist independent random variables $Z$ and $T$ in $P(V)$ such that for any $t\in ]0,1[$, we have:
\begin{eqnarray}\p (|\langle K_n\cdot v_\rho,AU_n^{-1}\cdot v_\rho\rangle|\leq t^n ) & \leq & \E \left(\phi_{t^n} ([K_n\cdot v_\rho],[{U_n}^{-1}\cdot v_\rho]) \right) \label{ch2awwalmarra}\\
& \leq & \E \left( \phi_{t^n} (Z,T) \right) + ||\phi_{t^n}|| \rho^n \\
& \leq & \p (|\langle Z,AT\rangle| \leq 2t^n ) + Max\{1,||A||\}\frac{\rho^n}{t^n} \label{ch2explique1}
 \end{eqnarray}

 In the last line, we confused between $Z$ and $T$ in $P(V)$ and some representative in the unit sphere. The bounds (\ref{ch2awwalmarra}) and (\ref{ch2explique1}) follow from (\ref{ch2yy}).\\
To prove our proposition, we can clearly suppose $t\in ]\rho,1[$. It suffices then to show that
$\p (|\langle Z,AT\rangle| \leq 2t^n )$  is sub-exponential.
The law of $T$ is the unique $\rho(\mu)^t$-invariant probability measure $\nu$  on $P(V)$ (Theorem \ref{ch2independence}). Moreover, a general lemma of Furstenberg (see for example \cite[Proposition 2.3 page 49]{bougerol}) shows that $\nu$ is proper. Hence, a.s. $AT\neq 0$. Moreover, we claim that following  the stronger statement holds
: there exist
 $C,\alpha>0$ such
that for every $t'\in ]0,1[$ and $n\in \N^*$:

 \begin{equation}\label{ch2tesada2}\p (||AT||\leq t'^n)\leq C t'^{n\alpha}\end{equation}
 Indeed, $A$ being a non zero endomorphism, there exist a non zero vector of norm one, $v_0$ such that $A^tv_0\neq 0$. Then by Theorem \ref{ch2hausdorff},
$$\p (||AT||\leq t'^n ) \;\leq\; \p(|\langle AT,v_0\rangle|\leq t'^n)\;\leq\; \p(|\langle T,A^tv_0\rangle|\leq t'^n)\;\leq\frac{C}{||A^tv_0||^\alpha} t'^{n\alpha}$$
Hence for every $t'\in ]t,1[$,
\begin{eqnarray}\p (|\langle Z,AT\rangle| \leq 2t^n ) &=& \p (|\langle Z,\frac{AT}{||AT||}\rangle| \leq 2\frac{t^n}{||AT||} ) \nonumber\\
&\leq& \p \left(|\langle Z,\frac{AT}{||AT||}\rangle| \leq 2 (t/t')^n \right)+\frac{C}{||A^tv_0||^\alpha}{t'}^{n\alpha} \nonumber\\
&\leq& Sup\{\p \left(\delta(Z,[H])\leq 2 (t/t')^n \right);\textrm{\;\;$H$ hyperplane of $V$}\} + C{t'}^
{n\alpha}\nonumber\end{eqnarray}
The last line is by independence of $Z$ and $T$. Theorem \ref{ch2hausdorff} shows that it decreases exponentially fast to zero.
\end{proof}

As an application, we give the

\begin{proof}[Proof of Theorem \ref{ch2tr1}]
Lemma \ref{ch2lemma} allows us to be in the situation of Theorem \ref{ch2theo1}, i.e., we have
 a representation $(\rho,V)$ whose irreducible sub-representations $\rho_1,\cdots, \rho_r$ are pairwise non isomorphic, a endomorphism
  $A\in End(V_1)\oplus \cdots \oplus End(V_r)$ whose restriction to each $End(V_i)$ non zero such that $\mathcal{V}=\{g\in G; Tr(gA)=0\}$.
Lemma \ref{ch2lyapunovcone} allows us to distinguish a representation, say $\rho_1$, whose Lyapunov exponent is the biggest. Lemma \ref{ch2split} shows that this representation is proximal. It suffices
to apply Theorem \ref{ch2theo1}.  \end{proof}


\begin{proof}[Proof of Theorem \ref{ch2tr2}]
 For every $k\in \N$, let
 $Sym^k(\R^d)$ be the vector space of homogenous polynomials on $d$ variables of degree $k$. The group $SL_d(\R)$ acts on $Sym^k(\R^d)$ by the formula:
$g.P(X_1,\cdots,X_d)=P\left(g^{-1}(X_1,\cdots,X_d) \right)$ for every $g\in SL_d(\R)$, $P\in Sym^k(\R^d)$.
A known fact (see for example \cite{fulton}) is
that the action of $SL_d(\R)$ on $Sym^k(\R^d)$ is irreducible for every $k\in \N$. \\ Consider now a proper algebraic hypersurface $\widetilde{\mathcal{V}}$ of $\R^d$ defined over $\R$, a non zero vector $x$ of $\R^d$  and denote $\mathcal{V}=\{g\in SL_d(\R); gx\in \widetilde{\mathcal{V}}\}$. Let now $P$ be the polynomial that defines $\widetilde{\mathcal{V}}$, $k$ its degree. The polynomial $P$ can be seen as a vector in $V=\oplus_{i=0}^k {Sym^i(\R^d)}$. Let $\rho_i$ be the action of $SL_d(\R)$  on $Sym^i(\R^d)$.
If $P_i$ denotes
projection of $P$ on $Sym^i(\R^d)$, then
 ``$g x \in \mathcal{V} \;\Leftrightarrow\; P(gx)=0 \; \;\Leftrightarrow\; \sum_{i=0}^k {f_i(g^{-1})}=0$'' where
 $f_i(g)=\rho_i(g) (P_i) (x) \in C(\rho_i)$  (see Definition \ref{ch2matrix}). Moreover, the highest weight of $Sym^i(\R^d)$ is strictly bigger
 (for the natural order on $X(\mathbf{A})$ defined in Section \ref{ch2subcartan}) than the one of $Sym^{i-1}(\R^d)$, the ratio being the highest weight  of  the natural representation of $SL_d(\R)$ on $\R^d$. We can then apply Remark \ref{ch2nec} and Theorem \ref{ch2theo1} to the probability measure $\mu^{-1}$.
\end{proof}

An application of the results of Section \ref{ch2subproba} independent from Theorem \ref{ch2theo1} is the

\begin{proof}[Proof of Theorem \ref{ch2tr3}]

 If the identity component $\mathbf{H}^0$ of $\mathbf{H}$
 is reductive, then by Proposition \ref{ch2reductive}, there exists a rational representation $(\rho,V)$ of $\mathbf{G}$ such that
 the reductive group $\mathbf{H}^0$ fixes a non zero vector $x$ of $V$. By decomposing $\rho$ into irreducible sub-representations, one can
  assume $(\rho,V)$ to be irreducible. If $h_1,\cdots,h_r$ denote  the cosets of the finite group $H/H^0$, then we can write
$$\p (S_n \in H) \leq \sum_{i=1}^r {\p (S_n h_i^{-1} \cdot x = x)} \leq \sum_{i=1}^r {\p \left(||\rho(S_n)\frac{ h_i^{-1} \cdot x}{||x||}|| = 1\right)}  $$
 Since $G$ has no compact factors, $\rho(G)$ is non compact. In particular, $\rho(G_\mu)$
  is not contained in a compact subgroup of $SL(V)$ because compact subgroups of $SL(V)$ are algebraic and $\rho(G_\mu)$ is Zariski dense in $\rho(G)$. Hence we can apply
  Furstenberg theorem (\cite{Furst}) which shows that  $L_{\rho(\mu)}>0$ (see Definition \ref{ch2liapou}).
  Applying Le Page large deviations theorem (Theorem \ref{ch2page}) shows  that for every $i=1,\cdots, r$, $\p \left(||S_n\cdot (h_i^{-1} \cdot x)|| \leq exp(nL_{\rho(\mu)}/2) \right)$ decreases exponentially fast to zero.

If $\mathbf{H}^0$ is not reductive, then it contains a unipotent Zariski connected $\R$-subgroup $\mathbf{U}$ which is normal in $\mathbf{H}^0$. Hence $\mathbf{H}^0\subset N(\mathbf{U})$, where $N(\mathbf{U})$ is the normalizer of $\mathbf{U}$ in $\mathbf{G}$. By \cite[Corollary 3.9]{elements}, there is an $\R$-parabolic subgroup $\mathbf{P}$ of $\mathbf{G}$ such that $N(\mathbf{U}) \subset \mathbf{P}$. By \cite[Proposition 5.14]{gpesred}, $\mathbf{P}$ is conjugated to one of the standard  parabolic subgroups
$\mathbf{P}_\theta$, $\theta \subset \Pi$
  described in  Section \ref{ch2subparabolic}. Hence, by Lemma \ref{ch2lemmepar}, $\mathbf{P}_\theta$ fixes the line generated by the highest weight $x_\alpha$ of
 $(\rho_\alpha,V_\alpha)$ for every $\alpha\not\in \theta$. Fix such $\alpha$. Hence, $$\mathbf{H}^0\subset\{g\in \mathbf{G}^0; g\cdot [x_\alpha]=[x_\alpha]\}$$
As in the previous paragraph, denote by $h_1,\cdots, h_r$ the cosets of the finite group $H/H^0$. Hence,

\begin{equation}\p (S_n \in H)\leq \sum_{i=1}^r{\p(\rho_\alpha(S_n)[h_i^{-1}x_\alpha]=[x_\alpha])}\label{ch2intassa}\end{equation}
The representation $\rho_\alpha$ is $G$-irreducible hence by connectedness, strongly
 irreducible. Moreover, it is proximal because $\Theta_{\rho_\alpha}=\{\alpha\}$, its highest weight space is a line and $G$ has no compact factors. By Golsheild-Margulis theorem (Theorem \ref{ch2margulis}), $\rho_\alpha(\Gamma)$ is proximal. Hence we can apply Lemma \ref{ch2hyperplane} which proves the exponential decay of the probability \ref{ch2intassa}.
 \end{proof}

\section{Application to generic Zariski density and to free subgroups of linear groups}
\label{ch2subzariski}
\subsection{Statement of the results and commentaries}


Let $\mathbf{G}$ be a semi-simple algebraic group defined over $\R$ and $G$ its group of real points.

\begin{question}
Let $\Gamma$ be a Zariski dense subgroup of $G$. Is it true that two ``random'' elements   in $\Gamma$
generate a Zariski dense subgroup of $G$. \end{question}
A motivation for this question is the following
\begin{question}
By the  Tits alternative \cite{tits}, any Zariski dense subgroup $\Gamma$ of $G$ contains a Zariski dense free subgroup on two generators. A natural question is to see if this property is generic.
In \cite[Theorem 1.1]{aoun}, we proved that  two ``random'' elements in $\Gamma$ generate a free subgroup.
The  question that arises immediately is to see if the latter subgroup is Zariski dense. \label{ch2question2}\end{question}

In recent works of Rivin \cite{genericrivin}, he showed the following:
\begin{theo} \cite[Corollary 2.11]{genericrivin} Let $\mathbf{G}=\mathbf{SL_d}$ and $\Gamma=SL_d(\Z)$ for some $d\geq 3$.
 Consider the uniform probability measure on a finite symmetric generating set and denote by
$\{S_n, n\geq 0\}$ the associated  random walk. Then, for any $g\in \Gamma$,
there exists a constant $c(g)\in ]0,1[$ such that
$$\p (\langle g,S_n\rangle\textrm{is  Zariski dense}) \geq 1-c(g)^n $$
Moreover, $c(g)$ is effective. \label{ch2genericrivin}\end{theo}
Passing from the ``1.5 random subgroup'' in Theorem \ref{ch2genericrivin}
to the subgroup generated by two random elements is delicate since the constant $c(g)$ depends among others
things on the norm of $g$. \\

Using our Theorem \ref{ch2tr1}, we will prove the following

\begin{theo}\label{ch2generic1} Let $G$ be the group of real points of a semi-simple algebraic group defined and split over $\R$. Let $\Gamma_1,\Gamma_2$ be two Zariski dense subgroups of $G$.
 Then there exists probability measures $\mu_1$ and $\mu_2$ respectively
on $\Gamma_1$ and $\Gamma_2$ with an exponential moment such that for some $c\in ]0,1[$ and all large $n$,
 $$\p (\textrm{$\langle S_{1,n},S_{2,n}\rangle$ is  Zariski dense and free}) \geq 1-c^n$$ where  $\{S_{2,n}; n \geq 0\}$
and  $\{S_{2,n}, n \geq 0\}$ are two independent random walks on $\Gamma_1$ (resp. $\Gamma_2$) associated respectively to
$\mu_1$ and $\mu_2$. This implies that almost surely, for $n$ big enough, the subgroup
 $\langle S_{1,n},S_{2,n}\rangle$ is Zariski dense and free.
\end{theo}

When $\mathbf{G}=\mathbf{SL_2}$, a stronger statement holds. It will follow immediately from our result in \cite{aoun}.

\begin{theo}\label{ch2generic2} Let $\Gamma_1,\Gamma_2$ be two Zariski dense subgroups of $SL_2(\R)$.
Then for any probability measures $\mu_1$ and $\mu_2$  with an exponential moment whose support generates respectively  $\Gamma_1$ and $\Gamma_2$,
 there exists
$c\in ]0,1[$ such that
 $$\p (\textrm{$\langle S_{1,n},S_{2,n}\rangle$ is  Zariski dense}) \geq 1-c^n$$  \end{theo}

\begin{remarque} Let us compare  Theorem \ref{ch2generic1} with Rivin's Theorem \ref{ch2genericrivin}.
The  advantage of our method is that it allows us to consider two elements at random and not a ``1.5 random subgroup'',
 which is crucial to solve Question \ref{ch2question2}.
 Furthermore,  we do not necessarily consider arithmetic groups, neither finitely generated groups:
 any   Zariski dense subgroup $\Gamma$ works. In addition to that, the statement shows that Zariski density is generic
 for a pair of random elements taken   in two groups $\Gamma_1$ and $\Gamma_2$ not necessarily equal. \\
  However, the big inconvenient is that our constants are not effective unlike Rivin's. Our result can be applied to prove the ``1.5 random subgroup'' but is
  less interesting than Rivin results since we don't know if the uniform probability measure on a finite symmetric generating of $SL_d(\Z)$
  works. \\
 For $d=2$, Theorem \ref{ch2generic2} is more satisfying; there is no restrictions neither on $\mu_1$ nor $\mu_2$.
  \end{remarque}

\subsection{Proofs}
\begin{proof}[Proof of Theorem \ref{ch2generic2}]
A subgroup  of $SL_2(\R)$ is Zariski dense if and  only it is not virtually solvable. In particular, a free subgroup of $SL_2(\R)$ is
always Zariski dense. But in Theorem \cite[Theorem 2.11]{aoun}, we proved that with the same assumptions as in Theorem \ref{ch2generic2},
$\p(\textrm{$\langle S_{1,n},S_{2,n}\rangle$ is not free} )$ decreases exponentially fast.

\end{proof}

\begin{proof}[Proof of Theorem \ref{ch2generic1}]
The key point is the following
\begin{lemme}\cite[Lemma 6.8]{strongtits} Let $k$ be a field of characteristic zero, $\mathbf{G}$ be a semi-simple group defined over $k$,  $G=\mathbf{G}(k)$. Then there exists a proper
algebraic variety $\mathcal{W}$ of $\mathbf{G} \times \mathbf{G}$ defined over $k$ such that any pair of elements $x,y\in G$ generate a Zariski dense subgroup
unless $(x,y)\in \mathcal{W}(k)$. \label{ch2strongtits}\end{lemme}

By Lemma \ref{ch2lemma}, there exist a rational real representation $(\rho,V)$ of $\mathbf{G} \times \mathbf{G}$, an endomorphism $A\in End(V_1) \oplus \cdots \oplus End(V_r)$ such that \begin{equation}\label{ch2lkop}\mathcal{W}=\{(g,h)\in \mathbf{G}\times \mathbf{G};\;Tr\left(\rho(g,h)A\right) =0\}\end{equation} Let $\rho_1, \cdots, \rho_r$ the irreducible sub-representations of $\rho$. Since $\Gamma_1 \times \Gamma_2$ is Zariski dense in $\mathbf{G}\times \mathbf{G}$, the  proof of Lemma \ref{ch2lyapunovcone} shows that there exist two probability measures $\mu_1$ and $\mu_2$ respectively on $\Gamma_1$ and $\Gamma_2$, a permutation $\sigma$ of $\{1, \cdots, r\}$ such that $L_{\rho_{\sigma(i)}(\mu_1 \otimes \mu_2)} > L_{\rho_{\sigma(i+1)}(\mu_1 \otimes \mu_2)}$ for $i=1, \cdots, r$.
Let $T_n$ be the random walk $(S_{1,n},S_{2,n})$ on $\Gamma_1\times \Gamma_2$
(i.e. the one corresponding to the probability measure $\mu_1 \otimes \mu_2$.)
By Lemma \ref{ch2strongtits} and identity (\ref{ch2lkop}),

\begin{equation}\p (\textrm{$\langle S_{n,1},S_{n,2}\rangle$ is not Zariski dense in $G$}) \leq \p \Big(Tr\left( \rho(T_n) A \right)=0 \Big)\end{equation}

Theorem \ref{ch2theo1} shows that the latter quantity decreases exponentially fast to zero.
\end{proof}

\section{Open problems and questions}

\begin{itemize}
\item   It is interesting to see if the probabilistic methods we used can generalize Theorem \ref{ch2tr1}. More precisely, if $\mu$ is a probability measure  with an exponential moment and whose support generates a Zariski dense subgroup of the real points of a semi-simple algebraic group $\mathbf{G}$, is it true that for every proper algebraic subvariety $\mathcal{V}$ of $\mathbf{G}$, $$\limsup\big[ \p (S_n \in \mathcal{V})\big]^{\frac{1}{n}}< 1$$
      where  $S_n$ the random walk associated to $\mu$.

\item The same question for Theorem \ref{ch2generic1} (i.e. replace there exists by for all, and do not assume the semi-simple algebraic group $\mathbf{G}$ $\R$-split.)
\end{itemize}

\def\cprime{$'$} \def\cprime{$'$}

\end{document}